\documentclass[reqno]{amsart}

\usepackage{amsthm, amsmath, amssymb, amsfonts, graphicx}

\usepackage[english]{babel}

\usepackage{bbm}
\usepackage[foot]{amsaddr}

\usepackage[inline]{enumitem}

\usepackage{float}
\restylefloat{table}

\usepackage{graphicx}
\usepackage{caption}
\usepackage{subcaption}

\usepackage[colorlinks=true, pdfstartview=FitV, linkcolor=blue,
            citecolor=blue, urlcolor=blue]{hyperref}
\usepackage[usenames]{color} 

\usepackage[square,numbers,sort]{natbib}

\newtheorem{theorem}{Theorem}
\newtheorem{proposition}[theorem]{Proposition}
\newtheorem{lemma}[theorem]{Lemma}

\theoremstyle{definition}

\theoremstyle{remark}
\newtheorem{remark}[theorem]{Remark}

\numberwithin{equation}{section}
\numberwithin{theorem}{section}

\DeclareMathOperator*{\argmax}{arg\,max} 

\DeclareMathOperator{\E}{\mathbb{E}}

\title{Asymptotics of impulse control problem with multiplicative reward}
\author{Damian Jelito$^{\ast,\dagger}$}

\email{djelito@impan.pl}

\author{\L{}ukasz Stettner$^{\ast}$}
\address{$^{\ast}$Institute of Mathematics, Polish Academy of Sciences, Warsaw, Poland}
\thanks{$^{\dagger}$Corresponding author}
\email{l.stettner@impan.pl}

\makeatletter
\def\namedlabel#1#2{\begingroup
    #2%
    \def\@currentlabel{#2}%
    \phantomsection\label{#1}\endgroup
}
\makeatother

\begin{document}

\begin{abstract}
We consider a long-run impulse control problem for a generic Markov process with a multiplicative reward functional. We construct a solution to the associated Bellman equation and provide a verification result. The argument is based on the probabilistic properties of the underlying process combined with the Krein-Rutman theorem applied to the specific non-linear operator. Also, it utilises the approximation of the problem in the bounded domain and  with the help of the dyadic time-grid.


\bigskip
\noindent \textbf{MSC2020 subject classifications:} 93E20, 49J21, 49K21, 60J25 
\end{abstract}

\maketitle


\section{Introduction}

Impulse control constitutes a versatile framework for controlling real-life stochastic systems. In this type of control, a decision-maker determines intervention times and instantaneous after-intervention states of the controlled process. By doing so, one can affect a continuous time phenomenon in a discrete time manner. Consequently, impulse control attracted considerable attention in the mathematical literature; see e.g.~\cite{Rob1978,BenLio1984,Dav1993} for classic contributions and~\cite{PalSte2010,BayEmm2013,DufPiu2016,MenRob2018} for more recent results. In addition to generic mathematical properties, impulse control problems were studied with reference to specific applications including i.a. controlling exchange rates, epidemics, and portfolios with transaction costs; see e.g.~\cite{Kor1999,RunYas2018,PiuPla2020} and references therein.

When looking for an optimal impulse control strategy, one must decide on the optimality criterion. Recently, considerable attention was paid to the so-called risk-sensitive functional given, for any $\gamma\in \mathbb{R}$, by
\begin{equation}\label{eq:RSC}
    \mu^\gamma(Z):=\begin{cases}
    \frac{1}{\gamma}\ln \mathbb{E}[\exp(\gamma Z)], & \gamma\neq 0,\\
    \mathbb{E}[Z], & \gamma =0,
    \end{cases}
\end{equation}
where $Z$ is a (random) payoff corresponding to a chosen control strategy; see~\cite{HowMat1972} for a seminal contribution. This functional with $\gamma=0$ corresponds to the usual linear criterion and the case $\gamma<0$ is associated with risk-averse preferences; see~\cite{BiePli2003} for a comprehensive overview. Also, the functional with $\gamma>0$ could be linked to the asymptotics of the power utility function; see~\cite{Ste2011b} for details. Recent comprehensive discussion on the long-run version with $\mu^\gamma$ could be found in~\cite{BisBor2023}. We refer also to~\cite{Pha2015} and references therein for a discussion on the connection between~\eqref{eq:RSC} and the duality of the large deviations-based criteria.

In this paper we focus on the use of the functional $\mu^\gamma$  with $\gamma>0$. More specifically, we consider the impulse control problem for some continuous time Markov process and construct a solution to the associated Bellman equation which  characterises an optimal impulse control strategy. To do this, we study the family of impulse control problems in bounded domains and then extend the analysis to the generic locally compact state space. This idea was used in~\cite{AraBis2018}, where PDEs techniques were applied to obtain the characterisation of the controlled diffusions in the risks-sensitive setting. A similar approximation for the the average cost per unit time problem was considered in~\cite{Ste2022}. 

The main contribution of this paper is a construction of a solution to the Bellman equation associated with the problem, see Theorem~\ref{th:Bellman_existence} for details. It should be noted that we get a bounded solution even though the state space could be unbounded and we assume virtually no ergodicity conditions for the uncontrolled process. Also, note that present results for $\gamma>0$ complement our recent findings on the impulse control with the risk-averse preferences; see~\cite{PitSte2019} for the dyadic case and~\cite{JelPitSte2019b} for the continuous time framework. Nevertheless, it should be noted that the techniques for $\gamma<0$ and $\gamma>0$ are substantially different and it is not possible to directly transform the results in one framework to the other; see e.g.~\cite{Nag2007b,JelPitSte2019a} for further discussion. 

The structure of this paper is as follows. In Section~\ref{S:preliminaries} we formally introduce the problem, discuss the assumptions and, in Theorem~\ref{th:ver_th}, provide a verification argument. Next, in Section~\ref{S:dyadic_bounded} we consider an auxiliary dyadic problem in a bounded domain and in Theorem~\ref{th:Krein_Rutman_use} we construct a solution to the corresponding Bellman equation. This is used in Section~\ref{S:dyadic} where we extend our analysis to the unbounded domain with the dyadic time-grid; see Theorem~\ref{th:w_delta_existence}. Next, in Section~\ref{S:asymptotics} we finally construct a solution to the Bellman equation for the original problem; see Theorem~\ref{th:Bellman_existence}. Finally, in Appendix~\ref{S:stopping} we discuss some properties of the optimal stopping problems that are used in this paper.


\section{Preliminaries}\label{S:preliminaries}

Let $X=(X_t)_{t\geq 0}$ be a continuous time standard Feller--Markov process on a filtered probability space $(\Omega, \mathcal{F}, (\mathcal{F}_t), \mathbb{P})$. The process $X$ takes values in a locally compact separable metric space $E$ endowed with a metric $\rho$ and the Borel $\sigma$-field $\mathcal{E}$. With any $x\in E$ we associate a probability measure $\mathbb{P}_x$ describing the evolution of the process $X$ starting in $x$; see Section 1.4 in~\cite{Shi1978} for details. Also, we use $\mathbb{E}_x$, $x\in E$, and  $P_t(x,A):=\mathbb{P}_x[X_t\in A]$, $t\geq 0$, $x\in E$, $A\in \mathcal{E}$, for the corresponding expectation operator and the transition probability, respectively. By $\mathcal{C}_b(E)$ we denote the family of continuous bounded real-valued functions on $E$. Also, to ease the notation, by $\mathcal{T}$, $\mathcal{T}_x$, and $\mathcal{T}_{x,b}$ we denote the families of stopping times, $\mathbb{P}_x$ a.s. finite stopping times, and $\mathbb{P}_x$ a.s. bounded stopping times, respectively. Also, for any $\delta>0$, by $\mathcal{T}^\delta\subset\mathcal{T}$, $\mathcal{T}_x^\delta\subset \mathcal{T}_x$, and $\mathcal{T}_{x,b}^\delta\subset \mathcal{T}_{x,b}$, we denote the respective subfamilies of dyadic stopping times, i.e. those taking values in the set $\{0,\delta, 2\delta, \ldots\}\cup \{\infty\}$.


Throughout this paper we fix some compact $U\subseteq E$ and we assume that a decision-maker is allowed to shift the controlled process to $U$. This is done with the help of an impulse control strategy, i.e. a sequence $V:=(\tau_i,\xi_i)_{i=1}^\infty$, where $(\tau_i)$ is an increasing sequence of stopping times and $(\xi_i)$ is a sequence of $\mathcal{F}_{\tau_i}$-measurable after-impulse states with values in $U$. With any starting point $x\in E$ and a strategy $V$ we associate a probability measure $\mathbb{P}_{(x,V)}$ for the controlled process $Y$. Under this measure, the process starts at $x$ and follows its usual (uncontrolled) dynamics up to the time $\tau_1$. Then, it is immediately shifted to $\xi_1$ and starts its evolution again, etc. More formally, we consider a countable product of filtered spaces $(\Omega, \mathcal{F}, (\mathcal{F}_t))$ and a coordinate process $(X_t^1, X_t^2, \ldots)$. Then, we define the controlled process $Y$ as $Y_t:=X_t^i$, $t\in [\tau_{i-1},\tau_i)$ with the convention $\tau_0\equiv 0$. Under the measure $\mathbb{P}_{(x,V)}$ we get $Y_{\tau_i}=\xi_i$; we refer to Chapter V in~\cite{Rob1978} for the construction details; see also Appendix in~\cite{Chr2014} and Section 2 in~\cite{Ste1982}. 
A strategy $V=(\tau_i,\xi_i)_{i=1}^\infty$ is called admissible if for any $x\in E$ we get $\mathbb{P}_{(x,V)}[\lim_{n\to\infty}\tau_n=\infty]=1$. The family of admissible impulse control strategies is denoted by $\mathbb{V}$. Also, note that, to simplify the notation, by $Y_{\tau_i^-}:=X_{\tau_i}^i$, $i\in \mathbb{N}_{*}$, we denote the state of the process right before the $i$th impulse (yet, possibly, after the jump).

In this paper we study the asymptotics of the impulse control problem given by
\begin{equation}\label{eq:imp_control_long_run3}
\sup_{V\in \mathbb{V}}J(x,V), \quad x\in E,
\end{equation}
where, for any $x\in E$ and $V\in \mathbb{V}$, we set
\begin{equation}\label{eq:J_funct}
J(x,V) :=\liminf_{T\to\infty} \frac{1}{T} \ln \mathbb{E}_{(x,V)}\left[e^{\int_0^T f(Y_s)ds+\sum_{i=1}^\infty 1_{\{\tau_i\leq T\}}c(Y_{\tau_i^-},\xi_i)}\right],
\end{equation}
with $f$ denoting the \textit{running cost function} and $c$ being the \textit{shift-cost function}, respectively. Note that this could be seen as a long-run standardised version of the functional~\eqref{eq:RSC} with $\gamma>0$ applied to the impulse control framework. Here, the standardisation refers to the fact that we do not use directly the parameter $\gamma$ (apart from its sign). Also, the problem is of the long-run type, i.e. the utility is averaged over time which improves the stability of the results.

The analysis in this paper is based on the approximation of the problem in a bounded domain. Thus, we fix a sequence $(B_m)_{m\in \mathbb{N}}$ of compact sets satisfying $B_m\subset B_{m+1}$ and $E=\bigcup_{m=0}^\infty B_m$. Also, we assume that $U\subset B_0$. Next, we assume the following conditions.

\begin{enumerate}
\item[(\namedlabel{A1}{$\mathcal{A}1$})] (Cost functions). The map $f:E\mapsto \mathbb{R}_{-}$ is a continuous and bounded. Also, the map $c:E\times U \mapsto \mathbb{R}_{-}$ is  continuous, bounded, and strictly non-positive, and satisfies the triangle inequality, i.e. for some $c_0<0$, we have
\[
0>c_0\geq c(x,\xi)\geq c(x,\eta)+c(\eta,\xi), \quad x\in E,\, \xi,\eta\in U.
\]
Also, we assume that $c$ satisfies the \textit{uniform limit at infinity} condition
\begin{equation}\label{eq:unif_c2}
\lim_{\Vert x\Vert, \Vert y\Vert \to \infty}\sup_{\xi \in U} |c(x,\xi)-c(y,\xi)|=0.
\end{equation}
\item[(\namedlabel{A2}{$\mathcal{A}2$})] (Transition probability continuity). For any $t$, the transition probability $P_t$ is continuous with respect to the total variation norm, i.e. for any sequence $(x_n)\subset E$ converging to $x\in E$, we get
\[
\lim_{n\to\infty}\sup_{A\in \mathcal{E}}|P_t(x_n,A)-P_t(x,A)|=0.
\]
\item[(\namedlabel{A3}{$\mathcal{A}3$})] (Distance control). For any compact set $\Gamma\subset E$, $t_0>0$, and $r_0>0$, we have
\begin{equation}\label{Co2}
\lim_{r\to\infty}M_{\Gamma}(t_0,r)=0,\qquad \lim_{t\to 0} M_{\Gamma}(t,r_0) =0,
\end{equation}
where $M_{\Gamma}(t,r):= \sup_{x\in \Gamma} \mathbb{P}_x[\sup_{s\in [0,t]} \rho(X_s,x)\geq r]$, $t,r>0$.
\item[(\namedlabel{A4}{$\mathcal{A}4$})] (Process irreducibility). For any $m\in \mathbb{N}$, $x\in B_m$, $\delta>0$, and any open set $\mathcal{O}\subset B_m$, we have
\[
\mathbb{P}_x\left[\cup_{i=1}^\infty \{X_{i\delta}\in \mathcal{O}\}\right]=1.
\]
Also, we assume that for any $x\in E$, $\delta>0$, and $m\in \mathbb{N}$, we have
\begin{equation}
    \mathbb{P}_x[\tau_{B_m}<\infty]=1
\end{equation}
where $\tau_{B_m}:=\delta\inf\{k\in \mathbb{N}\colon X_{k\delta}\notin B_m\}$.
\end{enumerate}

Before we proceed, let us comment on these assumptions. First, note that~\eqref{A1} states typical cost-functions conditions. In particular, the non-positivity assumption for $f$ is merely a technical normalisation. Indeed, for a generic $\tilde{f}\in \mathcal{C}_b(E)$ we may set $f(\cdot):=\tilde{f}(\cdot)-\Vert \tilde{f}\Vert\leq 0$ to get
\begin{equation*}
J^{f}(x,V) = J^{\tilde{f}}(x,V)-\Vert \tilde{f}\Vert, \quad x\in E, \, V\in \mathbb{V},
\end{equation*}
where  $J^{f}$ denotes the version of the functional $J$ from~\eqref{eq:J_funct} corresponding to the running cost function $f$.

Second, Assumption~\eqref{A2} states that the transition probabilities $\mathbb{P}_t(x,\cdot)$ are continuous with respect to the total variation norm. Note that this directly implies that the transition semigroup associated to $X$ is strong Feller, i.e. for any  $t>0$ and a  bounded measurable map $h\colon E\mapsto \mathbb{R}$, the map $x\mapsto \mathbb{E}_x[h(X_t)]$ is continuous and bounded.

Third, Assumption~\eqref{A3} quantifies distance control properties of the underlying process. It states that, for a fixed time horizon, the process with a high probability stays close to its starting point and, with a fixed radius, with a high probability it does not leave the corresponding ball with a sufficiently short time horizon. Note that these properties are automatically satisfied if the transition semigroup is $\mathcal{C}_0$-Feller; see Proposition 2.1 in~\cite{PalSte2010} and Proposition 6.4 in~\cite{BasSte2018} for details.

Finally, Assumption~\eqref{A4} states a form of the irreducibility of the process $X$. It requires that the process visits a sufficiently rich family of sets with unit probability.

To solve~\eqref{eq:imp_control_long_run3}, we show the existence of a solution to the impulse control Bellman equation, i.e. a function $w\in \mathcal{C}_b(E)$ and a constant $\lambda\in \mathbb{R}$ satisfying
\begin{equation}\label{eq:Bellman_long_run4}
w(x)=\sup_{\tau\in \mathcal{T}_{x,b}} \ln \mathbb{E}_x\left[\exp\left(\int_0^\tau (f(X_s)-\lambda)ds +Mw(X_\tau)\right)\right], \quad  x\in E,
\end{equation}
where the operator $M$ is given by 
\[
Mh(x):=\sup_{\xi\in U}(c(x,\xi)+h(\xi)), \quad h\in \mathcal{C}_b(E), \, x\in E.
\]

We start with a simple observation giving a lower bound for the constant $\lambda$ from~\eqref{eq:Bellman_long_run4}. To do this, we define the semi-group type by
\begin{equation}
\label{eq:r(f)}
r(f):=\lim_{t\to\infty} \frac{1}{t}\ln \sup_{x\in E} \mathbb{E}_x\left[e^{\int_0^t f(X_s)ds}\right];
\end{equation}
see e.g. Proposition 1 in~\cite{Ste1989} for a discussion on the properties of $r(f)$.
\begin{lemma}\label{lm:lambda_bound}
Let $(w,\lambda)$ be a solution to~\eqref{eq:Bellman_long_run4}. Then, we get $\lambda\geq r(f)$.
\end{lemma}
\begin{proof}
From~\eqref{eq:Bellman_long_run4}, for any $T\geq 0$, we get
\[
w(x)\geq \ln \mathbb{E}_x\left[e^{\int_0^T (f(X_s)-\lambda)ds +Mw(X_T)}\right].
\]
Thus, using the boundedness of $w$ and $Mw$, we get
\[
\Vert w\Vert \geq \sup_{x\in E}\ln \mathbb{E}_x\left[e^{\int_0^T (f(X_s)-\lambda)ds }\right] -\Vert Mw\Vert .
\]
Consequently, dividing both hand-sides by $T$ and letting $T\to\infty$, we get $0\geq r(f-\lambda)$, which concludes the proof.
\end{proof}

Let us now link a solution to~\eqref{eq:Bellman_long_run4} with the optimal value and an optimal strategy for~\eqref{eq:imp_control_long_run3}. To ease the notation, we recursively define the strategy $\hat{V}:=(\hat\tau_i,\hat\xi_i)_{i=1}^\infty$ for $i\in \mathbb{N}\setminus\{0\}$ by
\begin{equation}\label{eq:ver_th:strategy}
\begin{cases}
\hat\tau_{i}&:=\inf\{t\geq \hat\tau_{i-1}\colon w(X_t^i)=Mw(X_t^i)\},\\
\hat\xi_{i}&:= \argmax_{\xi\in U}\left(c(X_{\hat\tau_{i}}^i,\xi)+w(\xi)\right)1_{\{\hat\tau_{i}<\infty\}}+\xi_0 1_{\{\hat\tau_{i}=\infty\}},
\end{cases}
\end{equation}
where $\hat\tau_0:=0$ and $\xi_0\in U$ is some fixed point. First, we show that $\hat{V}$ is a proper strategy.

\begin{proposition}\label{pr:adm}
The strategy $\hat{V}$ given by~\eqref{eq:ver_th:strategy} is admissible.
\end{proposition}
\begin{proof}
To ease the notation, we define $N(0,T):=\sum_{i=1}^\infty 1_{\{\hat\tau_i\leq T\}}$, $T\geq 0$. We fix some $T> 0$ and $x\in E$, and show that we get
\begin{equation}\label{eq:lm:adm:0}
\mathbb{P}_{(x,\hat{V})}[N(0,T)=\infty]=0. 
\end{equation}

Recalling~\eqref{eq:ver_th:strategy}, on the event $A:=\{\lim_{i\to\infty}\hat\tau_i<+\infty\}$, for any $n\in \mathbb{N}$, $n\geq 1$, we get $w(X_{\hat\tau_n}^n)=Mw(X_{\hat\tau_n}^n)=c(X_{\hat\tau_n}^n,X_{\hat\tau_n}^{n+1})+w(X_{\hat\tau_n}^{n+1})$. Also, recalling that $c(x,\xi)\leq c_0<0$, $x\in E$, $\xi\in U$, for any $n\in \mathbb{N}$, $n\geq 1$, we have $w(X_{\hat\tau_n}^{n+1})-w(X_{\hat\tau_n}^{n})=-c(X_{\hat\tau_n}^{n},X_{\hat\tau_n}^{n+1})\geq -c_0>0$. Using this observation and Assumption~\eqref{A3}, we estimate the distance between consecutive impulses which will be used to prove~\eqref{eq:lm:adm:0}. More specifically, for any $k,m\in \mathbb{N}$, $k,m\geq 1$, we get
\begin{multline}\label{eq:lm:adm:1}
\sum_{n=k}^{k+m-2}(w(X_{\hat\tau_n}^{n+1})-w(X_{\hat\tau_{n+1}}^{n+1}))+(w(X_{\hat\tau_{k+m-1}}^{k+m})-w(X_{\hat\tau_k}^{k+1})) \\
= w(X_{\hat\tau_k}^{k+1})+\sum_{n=k+1}^{k+m-1}(w(X_{\hat\tau_n}^{n+1})-w(X_{\hat\tau_n}^n))-w(X_{\hat\tau_k}^{k+1})\\
= \sum_{n=k+1}^{k+m-1}(w(X_{\hat\tau_n}^{n+1})-w(X_{\hat\tau_n}^n))\geq -(m-1)c_0;
\end{multline}
it should be noted that the specific values for $k$ and $m$ will be determined later.
Using the continuity of $w$ we may find $K>0$ such that $\sup_{x,y\in U}(w(x)-w(y))\leq K$. Let $m\in \mathbb{N}$ be big enough to get $-(m-1)\frac{c_0}{2}>K$. Thus, noting that $X_{\hat\tau_{k+m-1}}^{k+m}, X_{\hat\tau_k}^{k+1}\in U$, we have $(w(X_{\hat\tau_{k+m-1}}^{k+m})-w(X_{\hat\tau_k}^{k+1}))\leq K<-(m-1)\frac{c_0}{2}$. Consequently, recalling~\eqref{eq:lm:adm:1}, on $A$, we get
\begin{equation}\label{eq:lm:adm:2}
\sum_{n=k}^{k+m-2}(w(X_{\hat\tau_n}^{n+1})-w(X_{\hat\tau_{n+1}}^{n+1}))\geq -(m-1)\frac{c_0}{2}. 
\end{equation}

Recalling the compactness of $U$ and the continuity of $w$ we may find $r>0$ such that for any $x\in U$ and $y\in E$ satisfying $\rho(x,y)<r$ we get $|w(x)-w(y)|<-\frac{c_0}{2}$. Let us now consider the family of events
\begin{equation}\label{eq:lm:adm:3}
B_k:=\bigcap_{n=k}^{k+m-2}\{\rho(X_{\hat\tau_n}^{n+1},X_{\hat\tau_{n+1}}^{n+1})<r\}, \quad k\in \mathbb{N},\, k\geq 1,
\end{equation}
and note that, for any $k\in \mathbb{N}$, $k\geq 1$, on $B_k\cap A$ we have $\sum_{n=k}^{k+m-2}(w(X_{\hat\tau_n}^{n+1})-w(X_{\hat\tau_{n+1}}^{n+1}))< -(m-1)\frac{c_0}{2}$. Thus, recalling~\eqref{eq:lm:adm:2}, for any $k\in \mathbb{N}$, $k\geq 1$, we get $\mathbb{P}_{(x_0,\hat{V})}[ B_k\cap A]=0$ and, in particular, we have 
\begin{equation}\label{eq:lm:adm:3.1}
\mathbb{P}_{(x_0,\hat{V})}[B_k\cap\{N(0,T)=\infty\}]=0.
\end{equation}

Let us now show that $\limsup_{k\to\infty}\mathbb{P}_{(x_0,\hat{V})}[B_k^c\cap\{N(0,T)=\infty\}]=0$. 
Noting that $\{N(0,T)=\infty\}=\{\lim_{i\to\infty}\hat\tau_i\leq T\}$, for any $t_0>0$ and $k\in \mathbb{N}$, $k\geq 1$, we get
\begin{align}\label{eq:lm:admis:1}
&\mathbb{P}_{(x_0,\hat{V})}\left[B_k^c\cap\{N(0,T)=\infty\}\right] \nonumber\\
&\phantom{,}\leq \mathbb{P}_{(x_0,\hat{V})}\left[\left(\bigcup_{n=k}^{k+m-2}\{\rho(X_{\hat\tau_n}^{n+1},X_{\hat\tau_{n+1}}^{n+1})\geq r\}\cap \{\hat\tau_{n+1}-\hat\tau_n\leq t_0\}\right) \cap \{\lim_{i\to\infty}\hat\tau_i\leq T\}\right]\nonumber\\
&\phantom{,}\phantom{=}+\mathbb{P}_{(x_0,\hat{V})}\left[\left(\bigcup_{n=k}^{k+m-2}\{\rho(X_{\hat\tau_n}^{n+1},X_{\hat\tau_{n+1}}^{n+1})\geq r\}\cap \{\hat\tau_{n+1}-\hat\tau_n> t_0\}\right) \cap \{\lim_{i\to\infty}\hat\tau_i\leq T\}\right]\nonumber\\
 &\phantom{,}\leq \mathbb{P}_{(x_0,\hat{V})}\left[\bigcup_{n=k}^{k+m-2}\{\sup_{t\in [0,t_0]}\rho(X_{\hat\tau_{n}}^{n+1},X_{\hat\tau_n+t}^{n+1})\geq r\}\cap \{\lim_{i\to\infty}\hat\tau_i\leq T\}\right]\nonumber\\
&\phantom{,}\phantom{=}+\mathbb{P}_{(x_0,\hat{V})}\left[\bigcup_{n=k}^{k+m-2}\{\hat\tau_{n+1}-\hat\tau_n> t_0\}\cap \{\lim_{i\to\infty}\hat\tau_i\leq T\}\right].
\end{align}
Using Assumption~\eqref{A3}, for any $\varepsilon>0$, we may find $t_0>0$, such that
\begin{equation}\label{eq:lm:adm:3.5}
\sup_{x\in U} \mathbb{P}_x\left[\sup_{t\in [0,t_0]}\rho(X_0,X_t)\geq r\right]\leq \frac{\varepsilon}{m-1}.
\end{equation}
Thus, using the strong Markov property and noting that $X_{\hat\tau_n}^{n+1}\in U$, for any $k\in \mathbb{N}$, $k\geq 1$, we get
\begin{multline}\label{eq:lm:adm:4}
\mathbb{P}_{(x_0,\hat{V})}\left[\bigcup_{n=k}^{k+m-2}\{\sup_{t\in [0,t_0]}\rho(X_{\hat\tau_{n}}^{n+1},X_{\hat\tau_n+t}^{n+1})\geq r\}\cap \{\lim_{i\to\infty}\hat\tau_i\leq T\}\right]\\
\leq \sum_{n=k}^{k+m-2}\mathbb{P}_{(x_0,\hat{V})}\left[\{\sup_{t\in [0,t_0]}\rho(X_{\hat\tau_{n}}^{n+1},X_{\hat\tau_n+t}^{n+1})\geq r\}\cap \{\hat\tau_n\leq T\}\right]\\
=\sum_{n=k}^{k+m-2}\mathbb{P}_{(x_0,\hat{V})}\left[\{\hat\tau_n\leq T\} \mathbb{P}_{X_{\hat\tau_n}^{n+1}}\left[\sup_{t\in [0,t_0]}\rho(X_{0},X_{t})\geq r \right]\right]
\leq \varepsilon.
\end{multline}
Recalling that $\varepsilon>0$ was arbitrary, for any $k\in\mathbb{N}$, $k\geq 1$, we get
\begin{equation}\label{eq:lm:adm:5}
\mathbb{P}_{(x_0,\hat{V})}\left[\bigcup_{n=k}^{k+m-2}\{\sup_{t\in [0,t_0]}\rho(X_{\hat\tau_{n}}^{n+1},X_{\hat\tau_n+t}^{n+1})\geq r\}\cap \{\lim_{i\to\infty}\hat\tau_i\leq T\}\right]=0.
\end{equation}

Now, to ease the notation, let $C_k:=\bigcup_{n=k}^{\infty}\{\hat\tau_{n+1}-\hat\tau_n> t_0\}\cap \{\lim_{i\to\infty}\hat\tau_i\leq T\}$, $k\in \mathbb{N}$, $k\geq 1$, and note that $C_{k+1}\subset C_k$, $k\in \mathbb{N}$, $k\geq 1$. We show that
\[
\lim_{k\to\infty}\mathbb{P}_{(x_0,\hat{V})}\left[C_k\right]=0.
\]
For the contradiction, assume that $\lim_{k\to\infty}\mathbb{P}_{(x_0,\hat{V})}\left[C_k\right]>0$. Consequently, we get $\mathbb{P}_{(x_0,\hat{V})}\left[\bigcap_{k=1}^\infty C_k\right]>0$. Note that for any $\omega \in \bigcap_{k=1}^\infty C_k$ we have $\lim_{i\to\infty}\hat\tau_i(\omega)\leq T$. In particular, we may find $i_0\in \mathbb{N}$ such that for any $n\geq i_0$ we get $\hat\tau_{n+1}(\omega)-\hat\tau_n(\omega)\leq \frac{t_0}{2}$. This leads to the contradiction as from the fact that $\omega \in \bigcap_{k=1}^\infty C_k$ we also get
\[
\omega \in \bigcap_{k=1}^\infty \bigcup_{n=k}^{\infty}\{\hat\tau_{n+1}-\hat\tau_n> t_0\}\subset \bigcup_{n=i_0}^{\infty}\{\hat\tau_{n+1}-\hat\tau_n> t_0\}.
\]
Consequently, we get $\lim_{k\to\infty}\mathbb{P}_{(x_0,\hat{V})}\left[C_k\right]=0$ and, in particular, we get
\[
\limsup_{k\to\infty}\mathbb{P}_{(x_0,\hat{V})}\left[\bigcup_{n=k}^{k+m-2}\{\hat\tau_{n+1}-\hat\tau_n> t_0\}\cap \{\lim_{i\to\infty}\hat\tau_i\leq T\}\right]\leq \lim_{k\to\infty}\mathbb{P}_{(x_0,\hat{V})}\left[C_k\right]=0.
\]
Hence, recalling~\eqref{eq:lm:admis:1} and~\eqref{eq:lm:adm:5}, we get
\[
\limsup_{k\to\infty}\mathbb{P}_{(x_0,\hat{V})}\left[B_k^c\cap\{N(0,T)=\infty\}\right]=0.
\]
Thus, recalling~\eqref{eq:lm:adm:3.1}, for any $k\in \mathbb{N}$, $k\geq 1$, we obtain
\[
\mathbb{P}_{(x_0,\hat{V})}\left[N(0,T)=\infty\right]=\mathbb{P}_{(x_0,\hat{V})}\left[B_k^c\cap\{N(0,T)=\infty\}\right],
\]
and letting $k\to\infty$, we conclude the proof of~\eqref{eq:lm:adm:0}.
\end{proof}

Now, we show the verification result linking~\eqref{eq:Bellman_long_run4} with the optimal value and an optimal strategy for~\eqref{eq:imp_control_long_run3}.

\begin{theorem}\label{th:ver_th}
Let $(w,\lambda)$ be a solution to~\eqref{eq:Bellman_long_run4} with $\lambda>r(f)$. Then, we get
\[
\lambda = \sup_{V\in \mathbb{V}} J(x,V) = J(x,\hat{V}), \quad x\in E,
\]
where the strategy $\hat{V}$ is given by~\eqref{eq:ver_th:strategy}.
\end{theorem}
\begin{proof}
The proof is based on the argument from Theorem 4.4 in~\cite{JelPitSte2019b} thus we show only an outline. First, we show that $\lambda=J(x,\hat{V})$, $x\in E$, where the strategy $\hat{V}$ is given by~\eqref{eq:ver_th:strategy}. Let us fix $x\in E$. Then, combining the argument used in Lemma 7.1 in~\cite{BasSte2018} and Proposition~\ref{pr:v_stop}, we  get that the process 
\[
e^{\int_0^{\hat\tau_1\wedge T} (f(X_s^1)-\lambda )ds  +w(X^1_{\hat\tau_1\wedge T})}, \quad T\geq 0,
\]
is a $\mathbb{P}_{(x,\hat{V})}$-martingale. Noting that on the event $\{\hat\tau_{k+1}<T\}$ we get $w(X^{k+1}_{\hat\tau_{k+1}})=Mw(X^{k+1}_{\hat\tau_{k+1}})=c(X^{k+1}_{\hat\tau_{k+1}},\hat\xi_{k+1})+w(\hat\xi_{k+1})$, $k\in \mathbb{N}$, for any $n\in \mathbb{N}$ we recursively get
\begin{align}\label{eq:th:ver_th:1}
e^{w(x)} & = \E_{(x,\hat{V})} \left[ e^{\int_0^{\hat\tau_1\wedge T} (f(Y_s)-\lambda )ds  +w(X^1_{\hat\tau_1\wedge T})}\right] \nonumber\\
 & =\E_{(x,\hat{V})} \left[ e^{\int_0^{\hat\tau_1\wedge T} (f(Y_s)-\lambda )ds  +1_{\{\hat\tau_1< T\}} c(X^1_{\hat\tau_1},X^2_{\hat\tau_1})+1_{\{\hat\tau_1< T\}}w(X^2_{\hat\tau_1})+1_{\{\hat\tau_1\geq T\}} w(X_T^1) }\right] \nonumber\\
 & = \E_{(x,\hat{V})} \left[ e^{\int_0^{\hat\tau_n\wedge T} (f(Y_s)-\lambda )ds  +\sum_{i=1}^n 1_{\{\hat\tau_i< T\}} c(X^{i}_{\hat\tau_i},X^{i+1}_{\hat\tau_i})}\times\right. \nonumber\\
 & \phantom{=\E_{(x,\hat{V})} = =} \left. \times e^{\sum_{i=1}^n 1_{\{\hat\tau_{i-1}< T\leq \hat\tau_{i}\}} w(X_T^{i})+1_{\{\hat\tau_n<T\}} w(X_{\hat\tau_n}^{n+1})}\right].
\end{align}
Recalling Proposition~\ref{pr:adm} we get $\hat\tau_n\to\infty$ as $n\to\infty$. Thus, letting $n\to\infty$ in~\eqref{eq:th:ver_th:1} and using Lebesgue’s dominated convergence theorem we get
\[
e^{w(x)}=\E_{(x,\hat{V})} \left[ e^{\int_0^{T} (f(Y_s)-\lambda )ds  +\sum_{i=1}^\infty 1_{\{\hat\tau_i< T\}} c(X^{i}_{\hat\tau_i},X^{i+1}_{\hat\tau_i})+\sum_{i=1}^\infty 1_{\{\hat\tau_{i-1}< T\leq \hat\tau_{i}\}} w(X_T^{i})}\right].
\]
Thus, recalling the boundedness of $w$, taking the logarithm of both sides, dividing by $T$, and letting $T\to\infty$ we obtain
\[
\lambda = \liminf_{T\to\infty}\E_{(x,\hat{V})} \left[ e^{\int_0^{T} f(Y_s)ds  +\sum_{i=1}^\infty 1_{\{\hat\tau_i< T\}} c(X^{i}_{\hat\tau_i},X^{i+1}_{\hat\tau_i})}\right].
\]

Second, let us fix some $x\in E$ and an admissible strategy $V=(\xi_i,\tau_i)_{i=1}^\infty\in \mathbb{V}$. We show that $\lambda\geq J(x,V)$. Using the argument from Lemma 7.1 in~\cite{BasSte2018} and Proposition~\ref{pr:v_stop}, we get that the process 
\[
e^{\int_0^{\tau_1\wedge T} (f(X_s^1)-\lambda )ds  +w(X^1_{\tau_1\wedge T})}, \quad T\geq 0,
\]
is a $\mathbb{P}_{(x,V)}$-supermartingale. Noting that on the event $\{\tau_{k+1}<T\}$ we have
\[
w(X^{k+1}_{\tau_{k+1}})\geq Mw(X^{k+1}_{\tau_{k+1}})\geq c(X^{k+1}_{\tau_{k+1}},\xi_{k+1})+w(\xi_{k+1}), \quad k\in \mathbb{N},
\]
for any $n\in \mathbb{N}$ we recursively get
\begin{align}\label{eq:th:ver_th:2}
e^{w(x)} & \geq \E_{(x,{V})} \left[ e^{\int_0^{\tau_1\wedge T} (f(Y_s)-\lambda )ds  +w(X^1_{\tau_1\wedge T})}\right] \nonumber\\
 & \geq\E_{(x,{V})} \left[ e^{\int_0^{\tau_1\wedge T} (f(Y_s)-\lambda )ds  +1_{\{\tau_1< T\}} c(X^1_{\tau_1},X^2_{\tau_1})+1_{\{\tau_1< T\}}w(X^2_{\tau_1})+1_{\{\tau_1\geq T\}} w(X_T^1) }\right] \nonumber\\
 & \geq \E_{(x,{V})} \left[ e^{\int_0^{\tau_n\wedge T} (f(Y_s)-\lambda )ds  +\sum_{i=1}^n 1_{\{\tau_i< T\}} c(X^{i}_{\tau_i},X^{i+1}_{\tau_i})}\times\right. \nonumber\\
 & \phantom{=\E_{(x,\hat{V})} = =} \left. \times e^{\sum_{i=1}^n 1_{\{\tau_{i-1}< T\leq \tau_{i}\}} w(X_T^{i})+1_{\{\tau_n<T\}} w(X_{\tau_n}^{n+1})}\right].
\end{align}
Recalling the admissibility of $V$, we get $\tau_n\to\infty$ as $n\to\infty$. Thus, letting $n\to\infty$ in~\eqref{eq:th:ver_th:2} and using Fatou's lemma, we get
\[
e^{w(x)}\geq \E_{(x,{V})} \left[ e^{\int_0^{T} (f(Y_s)-\lambda )ds  +\sum_{i=1}^\infty 1_{\{\tau_i< T\}} c(X^{i}_{\tau_i},X^{i+1}_{\tau_i})+\sum_{i=1}^n 1_{\{\tau_{i-1}< T\leq \tau_{i}\}} w(X_T^{i})}\right].
\]
Thus, taking the logarithm of both sides, dividing by $T$, and letting $T\to\infty$, we get
\[
\lambda \geq \liminf_{T\to\infty}\E_{(x,{V})} \left[ e^{\int_0^{T} f(Y_s)ds  +\sum_{i=1}^\infty 1_{\{\tau_i< T\}} c(X^{i}_{\tau_i},X^{i+1}_{\tau_i})}\right],
\]
which concludes the proof.
\end{proof}

In the following sections we construct a solution to~\eqref{eq:Bellman_long_run4}. In the construction we approximate the underlying problem using the dyadic time-grid. Also, we consider a version of the problem in the bounded domain.

\section{Dyadic impulse control in a bounded set}\label{S:dyadic_bounded}

In this section we consider a version of~\eqref{eq:imp_control_long_run3} with a dyadic-time-grid and obligatory impulses when the process leaves some compact set. In this way, we construct a solution to the bounded-domain dyadic counterpart of~\eqref{eq:Bellman_long_run4}. More specifically, let us fix some $\delta>0$ and $m\in \mathbb{N}$. We show the existence of a map $w_{\delta}^m\in \mathcal{C}_b(B_m)$ and a constant $\lambda_\delta^m\in \mathbb{R}$ satisfying
\begin{equation}\label{eq:w_delta_m:stopping_1step}
    w_{\delta}^m(x)=\sup_{\tau\in \mathcal{T}_{x,b}^{\delta}}\ln\mathbb{E}_x\left[e^{\int_0^{\tau\wedge \tau_{B_m}} (f(X_s)-\lambda_\delta^m )ds+Mw_\delta^m(X_{\tau\wedge \tau_{B_m}})}\right], \quad x\in B_m.
\end{equation}
In fact, we start with the analysis of an associated one-step equation. More specifically, we show the existence of a constant $\lambda_{\delta}^m\in \mathbb{R}$ and a map $w_{\delta}^m\in \mathcal{C}_b(B_m)$ satisfying
\begin{align}\label{eq:w_delta_m:2}
    w_{\delta}^m(x) &= \max\left(\ln\mathbb{E}_x\left[e^{\int_0^\delta (f(X_s)-\lambda_\delta^m)ds+1_{\{X_{\delta}\in B_m\}}w_{\delta}^m(X_{\delta})+1_{\{X_{\delta}\notin B_m\}}Mw_{\delta}^m(X_{\delta})}\right]\right.,\nonumber\\
    &\phantom{==} M w_{\delta}^m(x)\Big), \quad x\in B_m,\nonumber\\
    w_{\delta}^m(x) &= Mw_{\delta}^m(x), \quad x\notin B_m;
\end{align}
see Theorem~\ref{th:Krein_Rutman_use} for details. Also, note that we link~\eqref{eq:w_delta_m:2} with~\eqref{eq:w_delta_m:stopping_1step} in Theorem~\ref{th:w_delta_m:stopping}.

\begin{theorem}\label{th:Krein_Rutman_use}
There exists a constant $\lambda_{\delta}^m>0$ and a map $w_{\delta}^m\in \mathcal{C}_b(B_m)$ such that~\eqref{eq:w_delta_m:2} is satisfied and we get $\sup_{\xi\in U} w^m_\delta(\xi)=0$.
\end{theorem}
\begin{proof}
The idea of the proof is to use the Krein-Rutman theorem to get a positive eigenvalue with a non-negative eigenvector of the suitable operator. More specifically, we consider a cone of non-negative continuous and bounded functions $\mathcal{C}^+_b(B_m)\subset \mathcal{C}_b(B_m)$ and, for any $h\in \mathcal{C}^+_b(B_m)$, we define the operators
\begin{align*}
\tilde{M}h(x)&:=\sup_{\xi\in U} e^{c(x,\xi)}h(\xi), \quad x\in E,\\
\tilde{P}_{\delta}^m h(x)&:=\mathbb{E}_x\left[e^{\int_0^\delta f(X_s)ds}\left(1_{\{X_{\delta}\in B_m\}}h(X_{\delta})+1_{\{X_{\delta}\notin B_m\}}\tilde{M}h(X_{\delta})\right)\right], \quad x\in B_m, \\
\tilde{T}_{\delta}^m h(x)&:=\max\left(\tilde{P}_{\delta}^m h(x), \tilde{M} \tilde{P}_{\delta}^m h(x)\right), \quad x\in B_m.
\end{align*}
Now, we use the Krein-Rutman theorem to show that $\tilde{T}_{\delta}^m$ admits a positive eigenvalue and a non-negative eigenfunction; see Theorem 4.3 in~\cite{Bon1962} for details. We start with verifying the assumptions. First, note that $\tilde{T}_{\delta}^m$ is positively homogeneous, monotonic increasing, and we have
\[
\tilde{T}_{\delta}^m \mathbbm{1}(x)\geq e^{-\delta\Vert f\Vert-\Vert c\Vert }\mathbbm{1}(x), \quad x\in B_m,
\]
where $\mathbbm{1}$ denotes the function identically equal to $1$ on $B_m$. Also, using Assumption~\eqref{A2}, we get that $\tilde{T}_{\delta}^m$ transforms $\mathcal{C}^+_b(B_m)$ into itself and it is continuous with respect to the supremum norm. Let us now show that $\tilde{T}_{\delta}^m$ is in fact completely continuous. To see this, let $(h_n)_{n\in \mathbb{N}}\subset \mathcal{C}^+_b(B_m)$ be a bounded (by some constant $K>0$) sequence; using Arzel\`a-Ascoli Theorem we show that it is possible to find a convergent subsequence of $(\tilde{T}_{\delta}^m h_n)_{n\in \mathbb{N}}$. Note that, for any $n\in \mathbb{N}$, we get
\[
\Vert \tilde{T}_{\delta}^m h_n\Vert \leq e^{\delta\Vert f\Vert}K,
\]
hence $(\tilde{T}_{\delta}^m h_n)$ is uniformly bounded. Next, let us fix some $\varepsilon>0$, $x\in B_m$, and $(x_k)\subset B_m$ such that $x_k\to x$ as $k\to\infty$. Also, to ease the notation, for any $n\in \mathbb{N}$, we set $H_n(x):=1_{\{x\in B_m\}}h_n(x)+1_{\{x\notin B_m\}}\tilde{M}h_n(x)$, $x\in E$ and note that $H_n$ are measurable functions bounded by $2K$ uniformly in $n\in \mathbb{N}$. Then, for any $n,k\in \mathbb{N}$, we get
\begin{align}\label{eq:th:Krein-Rutman:1}
   | \tilde{T}_{\delta}^m h_n(x)-\tilde{T}_{\delta}^m h_n(x_k)|
   & \leq \left|\mathbb{E}_x\left[e^{\int_0^\delta f(X_s)ds}H_n(X_\delta)\right]  - \mathbb{E}_{x_k}\left[e^{\int_0^\delta f(X_s)ds}H_n(X_\delta)\right]\right| \nonumber\\
   &\phantom{=}+ |\tilde{M}\tilde{P}_{\delta}^m h_n(x)-\tilde{M}\tilde{P}_{\delta}^m h_n({x_k})|.
\end{align}
Also, using Assumption~\eqref{A1}, we may find $k\in \mathbb{N}$ big enough such that, for any  $n\in \mathbb{N}$, we obtain
\begin{equation}\label{eq:th:Krein-Rutman:2}
    |\tilde{M}\tilde{P}_{\delta}^m h_n(x)-\tilde{M}\tilde{P}_{\delta}^m h_n({x_k})|\leq e^{\delta\Vert f\Vert}K \sup_{\xi\in U} |e^{c(x,\xi)}-e^{c(x_k,\xi)}|\leq \frac{\varepsilon}{2}.
\end{equation}
Next, note that for any $u\in (0,\delta)$ and $n,k\in \mathbb{N}$, we get
\begin{align}\label{eq:th:Krein-Rutman:3}
    &\left|\mathbb{E}_x\left[e^{\int_0^\delta f(X_s)ds}H_n(X_\delta)\right]  - \mathbb{E}_{x_k}\left[e^{\int_0^\delta f(X_s)ds}H_n(X_\delta)\right]\right|  \nonumber\\
    & \phantom{\quad\quad}\leq \left|\mathbb{E}_x\left[\left(e^{\int_0^\delta f(X_s)ds}-e^{\int_u^\delta f(X_s)ds}\right)H_n(X_\delta)\right]\right| \nonumber\\ 
    & \phantom{\quad\quad =}+ \left|\mathbb{E}_{x_k}\left[\left(e^{\int_0^\delta f(X_s)ds}-e^{\int_u^\delta f(X_s)ds}\right)H_n(X_\delta)\right]\right|\nonumber\\
    & \phantom{\quad\quad =}+ \left|\mathbb{E}_{x_k}\left[e^{\int_u^\delta f(X_s)ds}H_n(X_\delta)\right]-\mathbb{E}_{x}\left[e^{\int_u^\delta f(X_s)ds}H_n(X_\delta)\right]\right|.
\end{align}
Also, using the inequality $|e^{y}-e^{z}|\leq e^{\max(y,z)}|y-z|$, $y,z\in \mathbb{R}$, we may find $u>0$ small enough such that, for any  $n,k\in \mathbb{N}$, we get
\begin{align}\label{eq:th:Krein-Rutman:4}
   \left|\mathbb{E}_{x_k}\left[\left(e^{\int_0^\delta f(X_s)ds}-e^{\int_u^\delta f(X_s)ds}\right)H_n(X_\delta)\right]\right|\leq 2Ke^{\delta \Vert f\Vert}u\Vert f\Vert\leq \frac{\varepsilon}{6}. 
\end{align}
Next, setting $F_n^u(x):=\mathbb{E}_x\left[e^{\int_0^{\delta-u}f(X_s)ds}H_n(X_{\delta-u})\right]$, $n\in \mathbb{N}$, $x\in E$, and using the Markov property combined with Assumption~\eqref{A2}, we may find $k\in \mathbb{N}$ big enough such that for any $n\in \mathbb{N}$, we get
\begin{multline*}
\left|\mathbb{E}_{x_k}\left[e^{\int_u^\delta f(X_s)ds}H_n(X_\delta)\right]-\mathbb{E}_{x}\left[e^{\int_u^\delta f(X_s)ds}H_n(X_\delta)\right]\right|\\
= |\mathbb{E}_{x_k}[F_n^u(X_u)]-\mathbb{E}_{x}[F_n^u(X_u)]|\\
\leq 2K e^{\delta \Vert f\Vert} \sup_{A\in \mathcal{E}}|P_u(x_k,A)-P_u(x,A)|\leq \frac{\varepsilon}{6}.
\end{multline*}
Thus, recalling~\eqref{eq:th:Krein-Rutman:3}--\eqref{eq:th:Krein-Rutman:4}, we get that for $k\in \mathbb{N}$ big enough and any $n\in \mathbb{N}$, we get $\left|\mathbb{E}_x\left[e^{\int_0^\delta f(X_s)ds}H_n(X_\delta)\right]  - \mathbb{E}_{x_k}\left[e^{\int_0^\delta f(X_s)ds}H_n(X_\delta)\right]\right| \leq \frac{\varepsilon}{2}$.
This combined with~\eqref{eq:th:Krein-Rutman:1}--\eqref{eq:th:Krein-Rutman:2} shows $
| \tilde{T}_{\delta}^m h_n(x)-\tilde{T}_{\delta}^m h_n(x_k)|\leq \varepsilon$
for $k\in \mathbb{N}$ big enough and any $n\in \mathbb{N}$, which proves the equicontinuity of the family $(\tilde{T}_{\delta}^m h_n)_{n\in \mathbb{N}}$. Consequently, using Arzel\`a-Ascoli, we may find a uniformly (in $x\in B_m$) convergent subsequence of $(\tilde{T}_{\delta}^m h_n)_{n\in \mathbb{N}}$ and the operator $\tilde{T}_{\delta}^m$ is  completely continuous. Thus, using the Krein-Rutman theorem we conclude that there exists a constant $\tilde\lambda_{\delta}^m>0$ and a non-zero map $h_{\delta}^m\in \mathcal{C}^+_b(B_m)$ such that
\begin{equation}\label{eq:th:Krein_Rutman_use:1}
    \tilde{T}_{\delta}^m h_{\delta}^m(x)=\tilde\lambda_{\delta}^m h_{\delta}^m(x), \quad x\in B_m.
\end{equation}
After a possible normalisation, we assume that $\sup_{\xi\in U}h_{\delta}^m(\xi)=1$. 

Let us now show that $h_{\delta}^m(x)>0$, $x\in B_m$. To see this, let us define $D:=e^{-\delta\Vert f\Vert} \frac{1}{\tilde\lambda_{\delta}^m}$ and let $\mathcal{O}_h\subset B_m$ be an open set such that
\begin{equation}\label{eq:O_h}
    \inf_{x\in O_h}h^m_{\delta}(x)>0;
\end{equation}
note that this set exists thanks to the continuity of $h_{\delta}^m$ and the fact that $h_{\delta}^m$ is non-zero. Next, using~\eqref{eq:th:Krein_Rutman_use:1}, we have
\[
h_{\delta}^m(x)\geq D\mathbb{E}_x\left[1_{\{X_\delta \in \mathcal{O}_h\}}h_{\delta}^m(X_\delta)+1_{\{X_\delta \in B_m\setminus\mathcal{O}_h\}}h_{\delta}^m(X_\delta)\right], \quad x\in B_m.
\]
Then, for any $n\in \mathbb{N}$, we inductively get
\begin{align*}
    h_{\delta}^m(x)&\geq D\mathbb{E}_x[1_{\{X_{\delta} \in \mathcal{O}_h \}}h_{\delta}^m(X_{\delta})] \\
    &\phantom{=}+\sum_{i=2}^n D^i\mathbb{E}_x\left[1_{\{X_{\delta} \in B_m\setminus \mathcal{O}_h,X_{2\delta} \in B_m\setminus \mathcal{O}_h, \ldots, X_{(i-1)\delta} \in B_m\setminus \mathcal{O}_h, X_{i\delta} \in \mathcal{O}_h \}}h_{\delta}^m(X_{i\delta})\right]\\
    &\phantom{=}+ D^n\mathbb{E}_x\left[1_{\{X_{\delta} \in B_m\setminus \mathcal{O}_h,X_{2\delta} \in B_m\setminus \mathcal{O}_h, \ldots, X_{i\delta} \in B_m\setminus\mathcal{O}_h\}}h_{\delta}^m(X_{n\delta})\right], \, x\in B_m.
\end{align*}
Thus, letting $n\to\infty$ and using Assumption~\eqref{A4} combined with~\eqref{eq:O_h}, we show $h_{\delta}^m(x)>0$ for any $x\in B_m$.

Next, we define $w^m_\delta(x):=\ln h^m_\delta(x)$, $x\in B_m$, and $\lambda^m_\delta:=\frac{1}{\delta}\ln \tilde\lambda^m_\delta$. Thus, from~\eqref{eq:th:Krein_Rutman_use:1}, we get that
the pair $(w^m_\delta, \lambda^m_\delta)$ satisfies 
\[
\tilde{T}_{\delta}^m e^{w_{\delta}^m}(x)=e^{\delta \lambda_{\delta}^m} e^{w_{\delta}^m(x)}, \quad x\in B_m, \quad \text{and}\quad \sup_{\xi\in U}w_{\delta}^m(\xi)=0.
\]
In fact, using Assumption~\eqref{A1} and the argument from Theorem 3.1 in~\cite{JelPitSte2019b}, we have 
\begin{align*}
    w_{\delta}^m(x) &= \max\left(\ln\mathbb{E}_x\left[e^{\int_0^\delta (f(X_s)-\lambda_\delta^m)ds+1_{\{X_{\delta}\in B_m\}}w_{\delta}^m(X_{\delta})+1_{\{X_{\delta}\notin B_m\}}Mw_{\delta}^m(X_{\delta})}\right]\right.,\nonumber\\
    &\phantom{==} M w_{\delta}^m(x)\Big), \quad x\in B_m.
\end{align*}
Finally, we extend the definition of $w^m_{\delta}$ to the full space $E$ by setting
\[
w^m_{\delta}(x):=Mw^m_{\delta}(x), \quad x\notin B_m;
\]
note that the definition is correct since, at the right-hand side, we need to evaluate $w^m_{\delta}$ only at the points from $U\subset B_0\subset B_m$ and this map is already defined there.
\end{proof}

As we show now, Equation~\eqref{eq:w_delta_m:2} may be linked to a specific martingale characterisation.

\begin{proposition}\label{pr:_delta_m:martingale}
Let $(w_{\delta}^m, \lambda^m_\delta)$ be a solution to~\eqref{eq:w_delta_m:2}. Then, for any $x\in B_m$, we get that the process
\[
z_{\delta}^m(n):=e^{\int_0^{(n\delta)\wedge \tau_{B_m}} (f(X_s)-\lambda_\delta^m )ds+w_\delta^m(X_{(n\delta)\wedge \tau_{B_m}})}, \quad n\geq 0,
\]
is a $\mathbb{P}_x$-supermartingale. Also, the process 
\[
z_{\delta}^m(n\wedge (\hat\tau^m_\delta/\delta)), \quad n\in \mathbb{N},
\]
is a $\mathbb{P}_x$-martingale, where $\hat\tau^m_\delta:=\delta\inf\{k\in \mathbb{N}\colon w_\delta^m(X_{k\delta})=M w_\delta^m(X_{k\delta})\}$.
\end{proposition}
\begin{proof}
To ease the notation, we show the proof only for $\delta=1$; the general case follows the same logic. Let us fix $m,n\in \mathbb{N}$ and $x\in B_m$. Then, using the fact $w_{1}^m(y)=Mw_{1}^m(y)$, $x\notin B_m$, and the inequality
\[
e^{w_{1}^m(y)} \geq \mathbb{E}_y\left[e^{\int_0^1 (f(X_s)-\lambda_1^m)ds+1_{\{X_{1}\in B_m\}}w_{1}^m(X_{1})+1_{\{X_{1}\notin B_m\}}Mw_{1}^m(X_{1})}\right], \quad y\in B_m,
\]
we have
\begin{align*}
    \mathbb{E}_x[z_{1}^m&(n+1)|\mathcal{F}_n]\\
    &=1_{\{\tau_{B_m}\leq n\}}e^{\int_0^{\tau_{B_m}} (f(X_s)-\lambda_1^m )ds+w_1^m(X_{ \tau_{B_m}})}\\
    &\phantom{=}+1_{\{\tau_{B_m}> n\}}e^{\int_0^{n}  (f(X_s)-\lambda_1^m )ds}\times\\
    &\phantom{=}\times\mathbb{E}_{X_n}[e^{\int_n^{n+1} (f(X_s)-\lambda_1^m )ds+1_{\{X_1\in B_m\}}w_1^m(X_{1})+1_{\{X_1\notin B_m\}}w_1^m(X_{1})}|\mathcal{F}_n]\\
    &=1_{\{\tau_{B_m}\leq n\}}e^{\int_0^{n\wedge\tau_{B_m}} (f(X_s)-\lambda_1^m )ds+w_1^m(X_{n\wedge \tau_{B_m}})}\\
    &\phantom{=}+1_{\{\tau_{B_m}> n\}}e^{\int_0^{n\wedge\tau_{B_m}}  (f(X_s)-\lambda_1^m )ds}\times\\
    &\phantom{=}\times\mathbb{E}_{X_n}[e^{\int_n^{n+1} (f(X_s)-\lambda_1^m )ds+1_{\{X_1\in B_m\}}w_1^m(X_{1})+1_{\{X_1\notin B_m\}}Mw_1^m(X_{1})}|\mathcal{F}_n]\\
    &\leq e^{\int_0^{n\wedge\tau_{B_m}}  (f(X_s)-\lambda_1^m )ds+w_1^m(X_{n\wedge \tau_{B_m}})} = z_1^m(n),
\end{align*}
which shows the supermartingale property of $(z_{1}^m(n))$. Next, note that on the set $\{\tau_{B_m}\wedge\hat\tau^m_1>n\}$ we get 
\[
e^{w_1^m(X_n)}=\mathbb{E}_{X_n}\left[e^{\int_0^1 (f(X_s)-\lambda_1^m)ds+1_{\{X_{1}\in B_m\}}w_{1}^m(X_{1})+1_{\{X_{1}\notin B_m\}}Mw_{1}^m(X_{1})}\right].
\]
Thus, we have
\begin{align*}
    \mathbb{E}_x[z_{1}^m((n+1)&\wedge \hat\tau^m_1)|\mathcal{F}_n]\\
    &=1_{\{\tau_{B_m}\wedge \hat\tau^m_1\leq n\}}e^{\int_0^{\tau_{B_m}\wedge \hat\tau^m_1} (f(X_s)-\lambda_1^m )ds+w_1^m(X_{ \tau_{B_m}\wedge \hat\tau^m_1})}\\
    &\phantom{=}+1_{\{\tau_{B_m}\wedge \hat\tau^m_1> n\}}e^{\int_0^{n}  (f(X_s)-\lambda_1^m )ds}\times\\
    &\phantom{=}\times\mathbb{E}_{X_n}[e^{\int_n^{n+1} (f(X_s)-\lambda_1^m )ds+1_{\{X_1\in B_m\}}w_1^m(X_{1})+1_{\{X_1\notin B_m\}}Mw_1^m(X_{1})}|\mathcal{F}_n]\\
    &= e^{\int_0^{n\wedge\tau_{B_m}\wedge\hat\tau^m_1}  (f(X_s)-\lambda_1^m )ds+w_1^m(X_{n\wedge \tau_{B_m}\wedge\hat\tau^m_1})} = z_1^m(n\wedge \hat\tau^m_1),
\end{align*}
which concludes the proof.
\end{proof}

Let us denote by $\mathbb{V}_{\delta,m}$ the family of impulse control strategies with impulse times in the time-grid $\{0,\delta, 2\delta, \ldots\}$ and obligatory impulses when the controlled process exits the set $B_m$ at some multiple of $\delta$.
Using a martingale characterisation of~\eqref{eq:w_delta_m:2}, we get that $\lambda^m_\delta$ is the optimal value of the impulse control problem with impulse strategies from $\mathbb{V}_{\delta,m}$. To show this result, we introduce a strategy $\hat{V}:=(\hat\tau_i,\hat\xi_i)_{i=1}^\infty\in \mathbb{V}_{\delta,m}$ defined recursively, for $i=1, 2,\ldots$, by
\begin{align}\label{eq:w^m_delta_strategy}
    \hat\tau_i&:=\hat\sigma_i\wedge \tau_{B_m}^i,\nonumber\\
    \hat\sigma_i&:=\delta \inf\{n\geq \hat\tau_{i-1}/\delta\colon n\in \mathbb{N}, \, w^m_\delta(X_{n\delta}^i)=Mw^m_\delta(X_{n\delta}^i)\},\nonumber\\
    \tau_{B_m}^i&:=\delta \inf \{n\geq \hat\tau_{i-1}/\delta\colon n\in \mathbb{N},\, X_{n\delta}^i\notin B_m\},\nonumber\\
    \hat\xi_i&:=\argmax_{\xi\in U}(c(X_{\hat\tau_i}^i,\xi)+w^m_\delta(\xi))1_{\{\hat\tau_{i}<\infty\}}+\xi_0 1_{\{\hat\tau_{i}=\infty\}},
\end{align}
where $\hat\tau_0:=0$ and $\xi_0\in U$ is some fixed point.

\begin{theorem}\label{th:lambda_delta_m}
Let $(w_{\delta}^m, \lambda^m_\delta)$ be a solution to~\eqref{eq:w_delta_m:2}. Then, for any $x\in B_m$, we get
\[
\lambda^m_\delta=\sup_{V\in \mathbb{V}_{\delta,m}}\liminf_{n\to\infty}\frac{1}{   n\delta}\ln \mathbb{E}_{(x,V)}\left[e^{\int_0^{n\delta} f(Y_s)ds+\sum_{i=1}^\infty 1_{\{\tau_i\leq n\delta\}}c(Y_{\tau_i^-},\xi_i)}\right].
\]
Also, the strategy $\hat{V}$ defined in~\eqref{eq:w^m_delta_strategy} is optimal. 
\end{theorem}
\begin{proof}
The proof follows the lines of the proof of Theorem~\ref{th:ver_th} and is omitted for brevity.
\end{proof}

Next, we link~\eqref{eq:w_delta_m:2} with an infinite horizon optimal stopping problem under the non-degeneracy assumption.

\begin{theorem}\label{th:w_delta_m:stopping}
Let $(w_{\delta}^m, \lambda^m_\delta)$ be a solution to~\eqref{eq:w_delta_m:2} with $\lambda^m_\delta>r(f)$. Then, we get that $(w_{\delta}^m, \lambda^m_\delta)$ satisfies~\eqref{eq:w_delta_m:stopping_1step}.
\end{theorem}
\begin{proof}
As in the proof of Proposition~\ref{pr:_delta_m:martingale}, we consider only $\delta=1$; the general case follows the same logic.

First, note that for any $x\in B_m$,  $n\in \mathbb{N}$, and $\tau\in \mathcal{T}^\delta_x$, using Proposition~\ref{pr:_delta_m:martingale} and Doob’s optional stopping theorem, we have
\[
e^{w_{1}^m(x)}\geq \mathbb{E}_x\left[e^{\int_0^{n\wedge\tau\wedge \tau_{B_m}} (f(X_s)-\lambda_1^m )ds+w_1^m(X_{n\wedge \tau\wedge \tau_{B_m}})}\right].
\]
Also, recalling the boundedness of $w^m_1$ and Proposition~\ref{pr:UI}, and letting $n\to\infty$, we get
\[
e^{w_{1}^m(x)}\geq \mathbb{E}_x\left[e^{\int_0^{\tau\wedge \tau_{B_m}} (f(X_s)-\lambda_1^m )ds+w_1^m(X_{ \tau\wedge \tau_{B_m}})}\right].
\]
Next, noting that $w_1^m(X_{\tau\wedge \tau_{B_m}})\geq Mw_1^m(X_{\tau\wedge \tau_{B_m}})$, and taking the supremum over $\tau\in \mathcal{T}^\delta_x$, we get
\[
e^{w_{1}^m(x)}\geq \sup_{\tau\in \mathcal{T}^\delta_x}\mathbb{E}_x\left[e^{\int_0^{\tau\wedge \tau_{B_m}} (f(X_s)-\lambda_1^m )ds+Mw_1^m(X_{ \tau\wedge \tau_{B_m}})}\right].
\]

Second, using again Proposition~\ref{pr:_delta_m:martingale}, for any $x\in B_m$ and  $n\in \mathbb{N}$, we get
\[
w_{1}^m(x)= \ln\mathbb{E}_x\left[e^{\int_0^{n\wedge\hat\tau^m_\delta \wedge \tau_{B_m}} (f(X_s)-\lambda_1^m )ds+w_1^m(X_{n\wedge \hat\tau^m_\delta \wedge \tau_{B_m}})}\right].
\]
Using again the boundedness of $w^m_1$ and Proposition~\ref{pr:UI}, and letting $n\to\infty$, we get
\[
w_{1}^m(x)= \ln\mathbb{E}_x\left[e^{\int_0^{\hat\tau^m_\delta \wedge \tau_{B_m}} (f(X_s)-\lambda_1^m )ds+w_1^m(X_{ \hat\tau^m_\delta \wedge \tau_{B_m}})}\right].
\]
In fact, noting that $w_1^m(X_{ \hat\tau^m_\delta \wedge \tau_{B_m}})=Mw_1^m(X_{ \hat\tau^m_\delta \wedge \tau_{B_m}})$, we obtain
\[
w_{1}^m(x)= \ln\mathbb{E}_x\left[e^{\int_0^{\hat\tau^m_\delta \wedge \tau_{B_m}} (f(X_s)-\lambda_1^m )ds+Mw_1^m(X_{ \hat\tau^m_\delta \wedge \tau_{B_m}})}\right],
\]
thus we get
\[
e^{w_{1}^m(x)}= \sup_{\tau\in \mathcal{T}^\delta_x}\mathbb{E}_x\left[e^{\int_0^{\tau\wedge \tau_{B_m}} (f(X_s)-\lambda_1^m )ds+Mw_1^m(X_{ \tau\wedge \tau_{B_m}})}\right].
\]
Finally, using Proposition~\ref{pr:stop_families}, we have
\[
e^{w_{1}^m(x)}= \sup_{\tau\in \mathcal{T}^\delta_{x,b}}\mathbb{E}_x\left[e^{\int_0^{\tau\wedge \tau_{B_m}} (f(X_s)-\lambda_1^m )ds+Mw_1^m(X_{ \tau\wedge \tau_{B_m}})}\right],
\]
which concludes the proof.
\end{proof}
\begin{remark}
In Theorem~\ref{th:w_delta_m:stopping} we showed that, if $\lambda^m_\delta>r(f)$, a solution to the one-step equation~\eqref{eq:w_delta_m:2} is uniquely characterised by the optimal stopping value function~\eqref{eq:w_delta_m:stopping_1step}. If $\lambda^m_\delta\leq r(f)$, the problem is degenerate and, in particular, we cannot use the uniform integrability result from Proposition~\ref{pr:UI}. In fact, in this case it is even possible that the one-step Bellman equation admits multiple solutions and the optimal stopping characterisation does not hold; see e.g. Theorem 1.13 in~\cite{PesShi2006} for details.
\end{remark}

\section{Dyadic impulse control}\label{S:dyadic}

In this section we consider a dyadic full-domain version of~\eqref{eq:imp_control_long_run3}. We construct a solution to the associated Bellman equation which will be later used to find a solution to~\eqref{eq:Bellman_long_run4}. The argument uses a bounded domain approximation from Section~\ref{S:dyadic_bounded}. More specifically, throughout this section we fix some $\delta>0$ and show the existence of a function $w_\delta\in \mathcal{C}_b(E)$ and a constant $\lambda_\delta\in \mathbb{R}$, which are a solution to the dyadic Bellman equation of the form
\begin{equation}\label{eq:Bellman_delta}
w_{\delta}(x)=\sup_{\tau\in \mathcal{T}_{x,b}^{\delta}}\ln\mathbb{E}_x\left[e^{\int_0^{\tau} (f(X_s)-\lambda_\delta )ds+Mw_\delta(X_{\tau})}\right], \quad x\in E.    
\end{equation}
In fact, we set
\begin{equation}\label{eq:lambda_delta}
    \lambda_{\delta}:=\lim_{m\to\infty} \lambda^m_\delta;
\end{equation}
note that this constant is well-defined as, from Theorem~\ref{th:lambda_delta_m}, recalling that $B_m\subset B_{m+1}$, we get $\lambda^m_\delta\leq \lambda^{m+1}_\delta$, $m\in \mathbb{N}$.

First, we state the lower bound for $\lambda_\delta$. 
\begin{lemma}\label{lm:lambda_delta_bound}
Let $(w_\delta,\lambda_\delta)$ be a solution to~\eqref{eq:Bellman_delta}. Then, we get $\lambda_\delta\geq r(f)$.
\end{lemma}
\begin{proof}
    The proof follows the lines of the proof of Lemma~\ref{lm:lambda_bound} and is omitted for brevity.
\end{proof}

Next, we show the existence of a solution to~\eqref{eq:Bellman_delta} under the non-degeneracy assumption $\lambda_\delta>r(f)$.

\begin{theorem}\label{th:w_delta_existence}
Let $\lambda_\delta$ be given by~\eqref{eq:lambda_delta} and  assume that $\lambda_\delta>r(f)$. Then, there exists $w_\delta\in \mathcal{C}_b(E)$ such that~\eqref{eq:Bellman_delta} is satisfied and we get $\sup_{\xi\in U} w_{\delta}(\xi)=0$.
\end{theorem}
\begin{proof}
We start with some general comments and an outline of the argument. First, note that from Theorem~\ref{th:Krein_Rutman_use}, for any $m\in \mathbb{N}$, we get a solution $(w_\delta^m,\lambda_\delta^m)$ to~\eqref{eq:w_delta_m:2} satisfying $\sup_{\xi\in U}w^m_\delta(\xi)=0$. Also, from the assumption $\lambda_\delta>r(f)$ we get $\lambda^m_\delta>r(f)$ for $m\in \mathbb{N}$ sufficiently big (for simplicity, we assume that $\lambda^0_\delta>r(f)$). Thus, using Theorem~\ref{th:w_delta_m:stopping}, we get that, for any $m\in \mathbb{N}$, the pair $(w_\delta^m,\lambda_\delta^m)$ satisfies~\eqref{eq:w_delta_m:stopping_1step}.

Second, to construct a function $w_\delta$, we use Arzel\`a-Ascoli Theorem. More specifically, recalling that $\sup_{\xi\in U}w^m_\delta(\xi)=0$ and using the fact that $-\Vert c\Vert\leq c(x,\xi)\leq 0 $, $x\in E$, $\xi\in U$, for any $m\in \mathbb{N}$ and $x\in E$, we get
\[
-\Vert c\Vert\leq Mw^m_\delta(x)\leq 0.
\]
Also, note that, for any $m\in \mathbb{N}$ and $x,y\in E$, we have
\[
|Mw^m_\delta(x)-Mw^m_\delta(y)|\leq \sup_{\xi \in U} |c(x,\xi)-c(y,\xi)|.
\]
Consequently, the sequence $(Mw^m_\delta)_{m\in \mathbb{N}}$ is uniformly bounded and equicontinuous. Thus, using Arzel\`a-Ascoli Theorem combined with a diagonal argument, we may find a subsequence (for brevity still denoted by $(Mw^m_\delta)_{m\in \mathbb{N}}$) and a map $\phi_\delta\in \mathcal{C}_b(E)$ such that $Mw^m_\delta(x)$ converges to $\phi_\delta(x)$ as $m\to\infty$ uniformly in $x$ from any compact set. In fact, using Assumption~\eqref{A1} and the argument from the first step of the proof of Theorem 4.1 in~\cite{JelPitSte2019b}, we get that the convergence is uniform in $x\in E$. Then, we define
\begin{equation}\label{eq:w_delta}
  w_{\delta}(x):=\sup_{\tau\in \mathcal{T}_{x,b}^{\delta}}\ln\mathbb{E}_x\left[e^{\int_0^{\tau} (f(X_s)-\lambda_\delta )ds+\phi_\delta(X_{\tau})}\right], \quad x\in E.
\end{equation}
To complete the construction, we show that $w^m_\delta$ converges to $w_\delta$ uniformly on compact sets. Indeed, in this case we have
\[
|Mw^m_\delta(x)-Mw_\delta(x)|\leq \sup_{\xi\in U}|w^m_\delta(\xi)-w_\delta(\xi)|\to 0, \quad m\to\infty,
\]
thus $\phi_\delta\equiv Mw_\delta$ and from~\eqref{eq:w_delta} we get that~\eqref{eq:Bellman_delta} is satisfied. Also, recalling that from Theorem~\ref{th:Krein_Rutman_use} we get $\sup_{\xi\in U} w^m_\delta(\xi)=0$, $m\in \mathbb{N}$, we also get $\sup_{\xi\in U} w_\delta(\xi)=0$.

Finally, to show the convergence, we define the auxiliary functions
\begin{align}
w_{\delta}^{m,1}(x)&:=\sup_{\tau\in \mathcal{T}_{x,b}^{\delta}}\ln\mathbb{E}_x\left[e^{\int_0^{\tau\wedge \tau_{B_m}} (f(X_s)-\lambda^m_\delta )ds+\phi_\delta(X_{\tau\wedge \tau_{B_m}})}\right], \quad x\in E,\label{eq:w_delta_m_1} \\
w_{\delta}^{m,2}(x)&:=\sup_{\tau\in \mathcal{T}_{x,b}^{\delta}}\ln\mathbb{E}_x\left[e^{\int_0^{\tau\wedge \tau_{B_m}} (f(X_s)-\lambda_\delta )ds+\phi_\delta(X_{\tau\wedge \tau_{B_m}})}\right], \quad x\in E \label{eq:w_delta_m_2}.
\end{align}
We split the rest of the proof into three steps: (1) proof that $|w_{\delta}^{m}(x)-w_{\delta}^{m,1}(x)|\to 0$ as $m\to\infty$ uniformly in $x\in E$; (2) proof that $|w_{\delta}^{m,1}(x)-w_{\delta}^{m,2}(x)|\to 0$ as $m\to\infty$ uniformly in $x\in E$; (3) proof that $|w_{\delta}^{m,2}(x)-w_{\delta}(x)|\to 0$ as $m\to\infty$ uniformly in $x$ from compact sets.

\medskip 

\textit{Step 1.} We show $|w_{\delta}^{m}(x)-w_{\delta}^{m,1}(x)|\to 0$ as $m\to\infty$ uniformly in $x\in E$. Note that, for any $x\in E$ and $m\in \mathbb{N}$, we have
\begin{align*}
w_\delta^{m,1}(x) 
& \leq  \sup_{\tau\in \mathcal{T}_{x,b}^\delta}\ln\left( \mathbb{E}_x\left[e^{\int_0^{\tau\wedge \tau_{B_m}} (f(X_s)-\lambda_\delta^m)ds +Mw_\delta^m(X_{\tau\wedge \tau_{B_m}})}\right]e^{\Vert \phi_\delta-Mw_\delta^m\Vert}\right)\\
& = w_\delta^m(x)+\Vert \phi_\delta-Mw_\delta^m\Vert.
\end{align*}
Similarly, we get $w_\delta^m(x)\leq w_\delta^{m,1}(x)+\Vert \phi_\delta-Mw_\delta^m\Vert$, thus
\[
\sup_{x\in E} |w_\delta^m(x)-w_\delta^{m,1}(x)|\leq \Vert \phi_\delta-Mw_\delta^m\Vert.
\]
Recalling the fact that $\phi_\delta$ is a uniform limit of $Mw_\delta^m$ as $m\to\infty$, we conclude the proof of this step.

\medskip 

\textit{Step 2.} We show that $|w_{\delta}^{m,1}(x)-w_{\delta}^{m,2}(x)|\to 0$ as $m\to\infty$ uniformly in $x\in E$. Recalling that $\lambda_\delta^m\uparrow \lambda_\delta$, we get $w_\delta^{m,1}(x)\geq w_\delta^{m,2}(x)\geq -\Vert \phi_\delta\Vert$, $x\in E$. Thus, using the inequality $|\ln y-\ln z|\leq \frac{1}{\min(y,z)}|y-z|$, $y,z>0$, we get
\begin{equation}\label{eq:lm:22:1}
0\leq w_\delta^{m,1}(x)-w_\delta^{m,2}(x)\leq e^{\Vert \phi_\delta\Vert} (e^{w_\delta^{m,1}(x)}-e^{w_\delta^{m,2}(x)}), \quad x\in E.
\end{equation}
Then, noting that $\phi_\delta(\cdot)\leq 0$, for any $m\in \mathbb{N}$ and $x\in E$, we obtain
\begin{align}\label{eq:lm:22:2}
0 \leq e^{w_\delta^{m,1}(x)}-e^{w_\delta^{m,2}(x)} &\leq \sup_{\tau\in \mathcal{T}_{x,b}^\delta}\left( \mathbb{E}_x\left[e^{\int_0^{\tau\wedge \tau_{B_m}} (f(X_s)-\lambda_\delta^m)ds +\phi_\delta(X_{\tau\wedge \tau_{B_m}})}\right]  \right.\nonumber\\
& \left.\phantom{=} - \mathbb{E}_x\left[e^{\int_0^{\tau\wedge \tau_{B_m}} (f(X_s)-\lambda_\delta)ds +\phi_\delta(X_{\tau\wedge \tau_{B_m}})}\right] \right)\nonumber\\
& \leq \sup_{\tau\in \mathcal{T}_{x,b}^\delta}\mathbb{E}_x\left[e^{\int_0^{\tau} f(X_s)ds}\left(e^{-\lambda_\delta^m\tau}- e^{-\lambda_\delta \tau}\right)\right].
\end{align}
Also, recalling that $\lambda_\delta^0\leq \lambda_\delta^m\leq \lambda_\delta$, $m\in \mathbb{N}$, for any $x\in E$ and $T\geq 0$, we get
\begin{align}\label{eq:lm:22:3}
0& \leq \sup_{\tau\in \mathcal{T}_{x,b}}\mathbb{E}_x\left[e^{\int_0^{\tau} f(X_s)ds}\left(e^{-\lambda_\delta^m \tau}- e^{\lambda_\delta\tau}\right)\right] \nonumber\\
& \leq \sup_{\substack{\tau\in \mathcal{T}_{x,b}}}\mathbb{E}_x\left[\left(1_{\{\tau\leq T\}}+1_{\{\tau>T\}}\right)e^{\int_0^{\tau} f(X_s)ds}\left(e^{-\lambda_\delta^m \tau}- e^{-\lambda_\delta \tau}\right)\right] \nonumber\\
& \leq \sup_{\substack{\tau< T\\ \tau\in \mathcal{T}_{x,b}}}e^{T\Vert f\Vert}\mathbb{E}_x\left[\left(e^{-\lambda_\delta^m \tau}- e^{-\lambda_\delta \tau}\right)\right]+\sup_{\substack{\tau\geq T\\ \tau\in \mathcal{T}_{x,b}}}\mathbb{E}_x\left[e^{\int_0^{\tau} (f(X_s)-\lambda_\delta^0)ds}\right].
\end{align}
Recalling  $\lambda_\delta^0>r(f)$ and using Lemma~\ref{lm:bound_r(f)}, for any $\varepsilon>0$, we may find $T\geq 0$, such that 
\[
0\leq \sup_{x\in E} \sup_{\substack{\tau\geq T\\ \tau\in \mathcal{T}_{x,b}}}\mathbb{E}_x\left[e^{\int_0^{\tau} (f(X_s)-\lambda_\delta^0)ds}\right]\leq \varepsilon .
\]
Also, using the inequality $|e^x-e^y|\leq e^{\max(x,y)}|x-y|$, $x,y\geq 0$, we obtain
\begin{multline}
    \sup_{\tau< T}\mathbb{E}_x\left[\left(e^{-\lambda_\delta^m \tau}- e^{-\lambda_\delta \tau}\right)\right]\\
    \leq \sup_{\tau< T} \mathbb{E}_x\left[ e^{\max(-\lambda_\delta^m \tau, -\lambda_\delta \tau)}\tau (\lambda_\delta-\lambda_\delta^m)\right]\\
    \leq e^{|\lambda_\delta^m | T} T(\lambda_\delta-\lambda_\delta^m).
\end{multline}
Thus, for fixed $T\geq 0$, we find $m\geq 0$, such that $e^{|\lambda_\delta^m | T} T(\lambda_\delta-\lambda_\delta^m)\leq \varepsilon $. Hence, recalling~\eqref{eq:lm:22:1}--\eqref{eq:lm:22:3}, for any $x\in E$ and $T,m$ big enough, we get
\[
0\leq w_\delta^{m,1}(x)-w_\delta^{m,2}(x)\leq  e^{\Vert \phi_\delta\Vert} 2\varepsilon.
\]
Recalling that $\varepsilon>0$ was arbitrary, we conclude the proof of this step.

\medskip 

\textit{Step 3.} We show that $|w_{\delta}^{m,2}(x)-w_{\delta}(x)|\to 0$ as $m\to\infty$ uniformly in $x$ from compact sets. First, we show that $w_\delta^{m,2}(x)\leq w_\delta(x)$ for any $m\in \mathbb{N}$ and $x\in E$. Let $\varepsilon>0$ and $\tau_m^\varepsilon\in \mathcal{T}_{x,b}^\delta$ be an $\varepsilon$-optimal stopping time for $w_\delta^{m,2}(x)$. Then, we get
\begin{align*}
w_\delta(x) & \geq \ln \mathbb{E}_x\left[e^{\int_0^{\tau_m^\varepsilon\wedge \tau_{B_m}} (f(X_s)-\lambda_\delta)ds +\phi_\delta(X_{\tau_m^\varepsilon\wedge \tau_{B_m}})}\right] \geq w_\delta^{m,2}(x)-\varepsilon .
\end{align*}
As $\varepsilon>0$ was arbitrary, we get $w_\delta^{m,2}(x)\leq w_\delta(x)$, $m\in \mathbb{N}$, $x\in E$. In fact, using a similar argument, for any $x\in E$, we may show that the map $m\mapsto w_{\delta}^{m,2}(x)$ is non-decreasing.

Second, let $\varepsilon>0$ and $\tau_\varepsilon\in \mathcal{T}_{x,b}^\delta$ be an $\varepsilon$-optimal stopping time for $w_\delta(x)$. 
Then, we obtain
\begin{align}\label{eq:th:dyadic:1}
0\leq w_\delta(x) - w_\delta^{m,2}(x)&\leq \ln \mathbb{E}_x\left[e^{\int_0^{\tau_\varepsilon} (f(X_s)-\lambda_\delta)ds +\phi_\delta(X_{\tau_\varepsilon})}\right]+\varepsilon \nonumber\\
&\phantom{=}  - \ln \mathbb{E}_x\left[e^{\int_0^{\tau_\varepsilon\wedge \tau_{B_m}} (f(X_s)-\lambda_\delta)ds +\phi_\delta(X_{\tau_\varepsilon\wedge \tau_{B_m}})}\right].
\end{align}
Noting that $\tau_{B_m}\uparrow +\infty$ as $m\to \infty$ and using the quasi left-continuity of $X$ combined with Lemma~\ref{pr:UI} and the boundedness of $\phi_\delta$, we get 
\[
\lim_{m\to\infty}  \mathbb{E}_x\left[e^{\int_0^{\tau_\varepsilon\wedge \tau_{B_m}} (f(X_s)-\lambda_\delta)ds +\phi_\delta(X_{\tau_\varepsilon\wedge \tau_{B_m}})}\right] = \mathbb{E}_x\left[e^{\int_0^{\tau_\varepsilon} (f(X_s)-\lambda_\delta)ds +\phi_\delta(X_{\tau_\varepsilon})}\right].
\]
Thus, using~\eqref{eq:th:dyadic:1} and recalling that $\varepsilon>0$ was arbitrary, we get $\lim_{m\to\infty} w_\delta^{m,2}(x) = w_\delta(x)$. Also, noting that by Proposition~\ref{pr:v_stop} and Proposition~\ref{pr:stop_families}, the maps $x\mapsto w_\delta(x)$ and $x\mapsto w_\delta^{m,2}(x)$ are continuous, and using the monotonicity of $m\mapsto w_\delta^{m,2}(x)$, from Dini's Theorem we get that $w_\delta^{m,2}(x)$ converges to $w_\delta(x)$ uniformly in $x$ from compact sets, which concludes the proof.
\end{proof}

We conclude this section with a verification result related to~\eqref{eq:Bellman_delta}.

\begin{theorem}\label{th:ver_th_delta}
Let $(w_\delta,\lambda_\delta)$ be a solution to~\eqref{eq:Bellman_delta} with $\lambda_\delta>r(f)$. Then, we get
\[
\lambda_\delta:=\sup_{V\in \mathbb{V}^\delta}\liminf_{n\to\infty} \frac{1}{n} \ln \mathbb{E}_{(x,V)}\left[e^{\int_0^{n\delta} f(Y_s)ds+\sum_{i=1}^\infty 1_{\{\tau_i\leq n\delta\}}c(Y_{\tau_i^-},\xi_i)}\right],
\]
where $\mathbb{V}^\delta$ is a family of impulse control strategies with impulse times on the dyadic time-grid $\{0,\delta, 2\delta, \ldots\}$.
\end{theorem}
\begin{proof}
    The proof follows the lines of the proof of Theorem~\ref{th:ver_th} and is omitted for brevity.
\end{proof}

\section{Existence of a solution to the Bellman equation}\label{S:asymptotics}
In this section we construct a solution $(w,\lambda)$ to~\eqref{eq:Bellman_long_run4}, which together with Theorem~\ref{th:ver_th} provides a solution to~\eqref{eq:imp_control_long_run3}. The argument uses a dyadic approximation and the results from Section~\ref{S:dyadic}. More specifically, we consider a family of dyadic time steps $\delta_k:=\frac{1}{2^k}$, $k\in \mathbb{N}$. First, we specify the value of $\lambda$. In fact, we define
\begin{equation}\label{eq:lambda}
  \lambda:=\liminf_{k\to\infty}\lambda_{\delta_k},  
\end{equation}
where $\lambda_{\delta_k}$ is a constant given by~\eqref{eq:lambda_delta}, corresponding to $\delta_k$. Note that, if for some $k_0\in \mathbb{N}$ we get $\lambda_{\delta_{k_0}}>r(f)$, then using Theorem~\ref{th:ver_th_delta}, we get that $\lambda_{\delta_k}\leq \lambda_{\delta_{k+1}}$, $k\geq k_0$, and the limit inferior could be replaced by the usual limit.

\begin{theorem}\label{th:Bellman_existence}
    Let $\lambda$ be given by~\eqref{eq:lambda} and   assume that $\lambda>r(f)$. Then, there exists $w\in \mathcal{C}_b(E)$ such that~\eqref{eq:Bellman_long_run4} is satisfied.
\end{theorem}
\begin{proof}
The argument is partially based on the one used in Theorem~\ref{th:w_delta_existence} thus we discuss only the main points. From the fact that $\lambda>r(f)$ we get $\lambda_{\delta_k}>r(f)$ for sufficiently big $k\in \mathbb{N}$; to simplify the notation, we assume $\lambda_{\delta_0}>r(f)$. Thus, using Theorem~\ref{th:w_delta_existence}, for any $k\in \mathbb{N}$, we get the existence of a map $w_{\delta_k}\in \mathcal{C}_b(E)$ satisfying
    \[
w_{\delta_k}(x)=\sup_{\tau\in \mathcal{T}_{x,b}^{\delta_k}}\ln\mathbb{E}_x\left[e^{\int_0^{\tau} (f(X_s)-\lambda_{\delta_k} )ds+Mw_{\delta_k}(X_{\tau})}\right], \quad x\in E
\]
and such that $\sup_{\xi\in U} w_{\delta_k}(\xi)=0$. Thus, we get
\[
-\Vert c\Vert\leq Mw_{\delta_k}(x)\leq 0, \quad k\in \mathbb{N},\, x\in E,
\]
and the family $(Mw_{\delta_k})_{k\in \mathbb{N}}$ is uniformly bounded. Also, it is equicontinuous as we have
\[
|Mw_{\delta_k}(x)-Mw_{\delta_k}(y)|\leq\sup_{x\in U} |c(x,\xi)-c(y,\xi)|, \quad x,y\in E.
\]
Thus, using Arzel\`a-Ascoli theorem, we may choose a subsequence (for brevity still denoted by $(Mw_{\delta_k})$), such that $(Mw_{\delta_k})$ converges uniformly on compact sets to some map $\phi$. In fact, using Assumption~\eqref{A1} and the argument from the first step of the proof of Theorem 4.1 from~\cite{JelPitSte2019b}, we get that $Mw_{\delta_k}(x)$ converges to $\phi(x)$ as $k\to\infty$ uniformly in $x\in E$. Next, let us define
\begin{equation}
    w(x):=\sup_{\tau\in \mathcal{T}_{x,b}}\ln\mathbb{E}_x\left[e^{\int_0^{\tau} (f(X_s)-\lambda )ds+\phi(X_{\tau})}\right], \quad x\in E.
\end{equation}
In the following, we show that $w_{\delta_k}$ converges to $w$ uniformly in compact sets as $k\to\infty$. Then, we get that $Mw_{\delta_k}$ converges to $Mw$, hence $Mw\equiv \phi$ and~\eqref{eq:Bellman_long_run4} is satisfied.

To show the convergence, we define
\[
w_{\delta_k}^1(x):=\sup_{\tau\in \mathcal{T}_{x,b}^{\delta_k}}\ln\mathbb{E}_x\left[e^{\int_0^\tau (f(X_s)-\lambda_{\delta_k})ds+\phi(X_\tau)}\right], \quad k\in \mathbb{N}, \, x\in E.
\]
In the following, we show that $|w(x)-w_{\delta_k}^1(x)|\to 0$ and $|w_{\delta_k}^1(x)-w_{\delta_k}(x)|\to 0$ as $k\to\infty$ uniformly in $x$ from compact sets. In fact, to show the first convergence, we note that
\[
w_{\delta_k}^0(x)\leq w_{\delta_k}^1(x)\leq w_{\delta_k}^2(x), \quad k\in \mathbb{N}, \,x\in E,
\]
where
\begin{align*}
    w_{\delta_k}^0(x)&:=\sup_{\tau\in \mathcal{T}_{x,b}^{\delta_k}}\ln\mathbb{E}_x\left[e^{\int_0^\tau (f(X_s)-\lambda)ds+\phi(X_\tau)}\right], \quad k\in \mathbb{N}, \, x\in E,\\
    w_{\delta_k}^2(x)&:=\sup_{\tau\in \mathcal{T}_{x,b}}\ln\mathbb{E}_x\left[e^{\int_0^\tau (f(X_s)-\lambda_{\delta_k})ds+\phi(X_\tau)}\right], \quad k\in \mathbb{N}, \, x\in E.
\end{align*}
Thus, to prove $|w(x)-w_{\delta_k}^1(x)|\to 0$ it is enough to show $|w(x)-w_{\delta_k}^0(x)|\to 0$ and $|w(x)-w_{\delta_k}^2(x)|\to 0$ as $k\to\infty$.

For transparency, we split the rest of the proof into three parts: (1) proof that $|w(x)-w_{\delta_k}^0(x)|\to 0$ as $k\to\infty$ uniformly in $x$ from compact sets; (2) proof that $|w(x)-w_{\delta_k}^2(x)|\to 0$ as $k\to\infty$ uniformly in $x\in E$; (3) proof that $|w_{\delta_k}^1(x)-w_{\delta_k}(x)|\to 0$ as $k\to\infty$ uniformly in $x\in E$. 

\medskip 

\textit{Step 1.} We show that $|w(x)-w_{\delta_k}^0(x)|\to 0$ as $k\to\infty$ as $k\to\infty$ uniformly in $x$ from compact sets. First, note that we have $w_{\delta_k}^0(x)\leq w(x)$, $k\in \mathbb{N}$, $x\in E$. Next, for any $x\in E$ and $\varepsilon>0$, let $\tau_\varepsilon\in \mathcal{T}_{x,b}$ be an $\varepsilon$-optimal stopping time for $w(x)$ and let $\tau_\varepsilon^k$ be its $\mathcal{T}^{\delta_k}_{x,b}$ approximation given by
\[
\tau_\varepsilon^k:=\inf \left\{\tau \in \mathcal{T}^{\delta_k}_{x,b}: \tau \geq \tau_\varepsilon\right\}=\sum_{j=1}^{\infty} 1_{\left\{\frac{j-1}{2^k}<\tau_\varepsilon \leq \frac{j}{2^m}\right\}} \frac{j}{2^k}.
\]
Then, we get
\begin{align*}
    0&\leq w(x)-w_{\delta_k}^0(x)\\
    &\leq \mathbb{E}_x\left[e^{\int_0^{\tau_\varepsilon} (f(X_s)-\lambda)ds+\phi(X_{\tau_\varepsilon})}\right]-\mathbb{E}_x\left[e^{\int_0^{\tau_\varepsilon^k} (f(X_s)-\lambda)ds+\phi(X_{\tau_\varepsilon^k})}\right]+\varepsilon.
\end{align*}
Also, using Proposition~\ref{pr:UI} and letting $k\to\infty$, we have
\[
\lim_{k\to\infty}\mathbb{E}_x\left[e^{\int_0^{\tau_\varepsilon^k} (f(X_s)-\lambda)ds+\phi(X_{\tau_\varepsilon^k})}\right]=\mathbb{E}_x\left[e^{\int_0^{\tau_\varepsilon} (f(X_s)-\lambda)ds+\phi(X_{\tau_\varepsilon})}\right].
\]
Consequently, recalling that $\varepsilon>0$ was arbitrary, we obtain $\lim_{k\to\infty}w_{\delta_k}^0(x)=w(x)$ for any $x\in E$. In fact, using the monotonicity of the sequence $(w_{\delta_k}^0)_{k\in \mathbb{N}}$ combined with Proposition~\ref{pr:v_stop}, Proposition~\ref{pr:stop_families}, and Dini's theorem, we get that the convergence is uniform on compact sets, which concludes the proof of this step.
\medskip 

\textit{Step 2.} We show that $|w(x)-w_{\delta_k}^2(x)|\to 0$ as $k\to\infty$ uniformly in $x\in E$. First, note that $-\Vert \phi\Vert\leq w(x)\leq w_{\delta_k}^2(x)$, $k\in \mathbb{N}$, $x\in E$. Thus, using the inequality $|\ln y-\ln z|\leq \frac{1}{\min(y,z)}|y-z|$, $y,z>0$, we get 
\[
0\leq w_{\delta_k}^2(x)-w(x)\leq e^{\Vert \phi\Vert}(e^{w_{\delta_k}^2(x)}-e^{w(x)}), \quad k\in \mathbb{N}, \, x\in E.
\]
Also, recalling that $\phi(\cdot)\leq 0$, for any $k\in \mathbb{N}$ and $x\in E$, we obtain
\[
0\leq e^{w_{\delta_k}^2(x)}-e^{w(x)}\leq \sup_{\tau\in \mathcal{T}_{x,b}}\mathbb{E}_x\left[e^{\int_0^\tau f(X_s) ds}\left(e^{-\lambda_{\delta_k}\tau}-e^{-\lambda\tau} \right) \right].
\]
Thus, repeating the argument from the second step of the proof of Theorem~\ref{th:w_delta_existence}, we get $w_{\delta_k}^2(x)\to w(x)$ as $k\to\infty$ uniformly in $x\in E$, which concludes the proof of this step.

\medskip 

\textit{Step 3.} We show that $|w_{\delta_k}^1(x)-w_{\delta_k}(x)|\to 0$ as $k\to\infty$ uniformly in $x\in E$. In fact, recalling that $\Vert Mw_{\delta_k}-\phi\Vert\to 0$ as $k\to\infty$, the argument follows the lines of the one used in the first step of the proof of Theorem~\ref{th:w_delta_existence}. This concludes the proof.
\end{proof}


\appendix

\section{Properties of optimal stopping problems}\label{S:stopping}

In this section we discuss some properties of the optimal stopping problems that are used in this paper. Throughout this section we consider $g,G \in \mathcal{C}_b(E)$ and assume $G(\cdot)\leq 0$ and $r(g)<0$, where $r(g)$ is the type of the semigroup given by~\eqref{eq:r(f)} corresponding to the map $g$. We start with a useful result related to the asymptotic behaviour of the running cost function $g$.

\begin{lemma}\label{lm:bound_r(f)}
Let $a$ be such that $r(g)<a<0$. Then,
\begin{enumerate}
\item The map $x\mapsto U_0^{g-a} 1(x):=\mathbb{E}_x\left[\int_0^\infty e^{\int_0^t (g(X_s) -a)ds}dt\right]$ is continuous and bounded.
\item We get 
\[
\lim_{T\to\infty}\sup_{x\in E}\sup_{\substack{\tau\geq T\\ \tau\in \mathcal{T}_x}}\mathbb{E}_x\left[e^{\int_0^\tau g(X_s)ds}\right]= 0.
\]
\end{enumerate}
\end{lemma}
\begin{proof} For transparency, we prove the claims point by point.

\medskip
\textsc{Proof of (1).} First, we show the boundedness of $x\mapsto U_0^{g-a} 1(x)$. Let $\varepsilon<a-r(g)$. Using the definition of $r(g-a)$ we may find $t_0\geq 0$, such that for any $t\geq t_0$ we get $\sup_{x\in E} \mathbb{E}_x\left[e^{\int_0^t (g(X_s)-a)ds}\right]\leq e^{t(r(g)-a+\varepsilon)}$. Then, using Fubini's theorem and noting that $r(g)-a+\varepsilon<0$, for any $x_0\in E$, we get
\begin{align*}
0\leq U_0^{g-a} 1(x_0)&\leq \int_0^\infty\sup_{x\in E}\mathbb{E}_x\left[ e^{\int_0^t (g(X_s) -a)ds}\right]dt\\& = \int_0^{t_0}\sup_{x\in E}\mathbb{E}_x\left[ e^{\int_0^t (g(X_s) -a)ds}\right]dt+\int_{t_0}^\infty\sup_{x\in E}\mathbb{E}_x\left[ e^{\int_0^t (g(X_s) -a)ds}\right]dt\\
&\leq \int_0^{t_0}e^{t(\Vert g\Vert-a)}dt+\int_{t_0}^\infty e^{t(r(g)-a+\varepsilon)}dt<\infty,
\end{align*}
which concludes the proof of the boundedness of $x\mapsto U_0^{g-a} 1(x)$.

For the continuity, note that using Assumption~\eqref{A2} and repeating the argument used in Lemma 4 in Section II.5 of~\cite{GikSko2004}, we get that $x\mapsto \mathbb{E}_x\left[ e^{\int_0^t (g(X_s) -a)ds}dt\right]$ is continuous for any $t\geq 0$. Also, as in the proof of the boundedness, we may show 
\[
0\leq \sup_{x\in E} \mathbb{E}_x\left[ e^{\int_0^t (g(X_s) -a)ds}\right] \leq e^{t(\Vert g\Vert-a)}1_{\{t\in [0,t_0]\}}+ e^{t(r(g)-a+\varepsilon)}1_{\{t>t_0\}}
\]
and the upper bound is integrable (with respect to $t$). Thus, using Lebesgue’s dominated convergence theorem, we get the continuity of the map $x\mapsto U_0^{g-a} 1(x)=\int_0^\infty\mathbb{E}_x\left[ e^{\int_0^t (g(X_s) -a)ds}\right]dt$, which concludes the proof of this step.

\medskip
\textsc{Proof of (2).} Noting that $U_0^{g-a} 1(x)\geq \int_0^1 e^{-t(\Vert g\Vert -a)}dt$, $x\in E$, we may find $d>0$, such that $U_0^{g-a}1(x)\geq d>0$, $x\in E$. Thus, recalling that $a<0$, we get
\begin{align*}
0&\leq \sup_{\substack{\tau\geq T \\ \tau\in \mathcal{T}_x}}\mathbb{E}_x\left[e^{\int_0^\tau g(X_s)ds}\right]\\
&\leq \sup_{\substack{\tau\geq T \\ \tau\in \mathcal{T}_x}}\mathbb{E}_x\left[e^{\int_0^\tau (g(X_s)-a)ds} e^{a\tau} U_0^{g-a} 1(X_{\tau})\frac{1}{d}\right]\\
& \leq \frac{e^{aT}}{d}\sup_{\substack{\tau\geq T \\ \tau\in \mathcal{T}_x}}\mathbb{E}_x\left[e^{\int_0^\tau (g(X_s)-a)ds}  U_0^{g-a} 1(X_{\tau})\right]\\
& =\frac{e^{aT}}{d}\sup_{\substack{\tau\geq T \\ \tau\in \mathcal{T}_x}}\mathbb{E}_x\left[\int_0^\infty e^{\int_0^{t+\tau} (g(X_s) -a)ds}dt\right]\\
& = \frac{e^{aT}}{d}\sup_{\substack{\tau\geq T \\ \tau\in \mathcal{T}_x}}\mathbb{E}_x\left[\int_{\tau}^\infty e^{\int_0^{t} (g(X_s) -a)ds}dt\right]\\
& \leq \frac{e^{aT}}{d}\mathbb{E}_x\left[\int_{0}^\infty e^{\int_0^{t} (g(X_s) -a)ds}dt\right]\leq \frac{e^{aT}}{d} \Vert U_0^{g-a} 1\Vert \to 0, \quad T\to \infty,
\end{align*}
which concludes the proof.
\end{proof}

Using Lemma~\ref{lm:bound_r(f)} we get the uniform integrability of a suitable family of random variables. This result is extensively used throughout the paper as it simplifies numerous limiting arguments.

\begin{proposition}\label{pr:UI}
For any $x\in E$, the family $\{e^{\int_0^\tau g(X_s)ds}\}_{\tau\in \mathcal{T}_x}$ is $\mathbb{P}_x$-uniformly integrable.
\end{proposition}
\begin{proof}
Let us fix some $x\in E$ and, for any $\tau\in \mathcal{T}_x$ and $n\in \mathbb{N}$, define the event $A_n^\tau:=\{\int_0^\tau g(X_s)ds\geq n\}$. Note that for any $T\geq 0$, we get
\begin{align*}
    \sup_{\tau\in \mathcal{T}_x} \mathbb{E}_x[1_{A_n^\tau} e^{\int_0^\tau g(X_s)ds}] & \leq \sup_{\substack{\tau \leq T \\ \tau\in \mathcal{T}_x}} \mathbb{E}_x[1_{A_n^\tau} e^{\int_0^\tau g(X_s)ds}] + \sup_{\substack{\tau> T \\ \tau\in \mathcal{T}_x}} \mathbb{E}_x[1_{A_n^\tau} e^{\int_0^\tau g(X_s)ds}]\\
    & \leq \sup_{\substack{\tau \leq T \\ \tau\in \mathcal{T}_x}}e^{T \Vert g\Vert} \mathbb{P}_x[A_n^\tau ] + \sup_{\substack{\tau> T \\ \tau\in \mathcal{T}_x}} \mathbb{E}_x[ e^{\int_0^\tau g(X_s)ds}].
\end{align*}
Next, for any $\varepsilon>0$, using Lemma~\ref{lm:bound_r(f)}, we may find $T>0$ big enough to get 
\[
\sup_{\substack{\tau> T \\ \tau\in \mathcal{T}_x}} \mathbb{E}_x[ e^{\int_0^\tau g(X_s)ds}]<\varepsilon.
\]
Also, noting that for $\tau\leq T$, we get $A_n^\tau \subset \{T\Vert g\Vert \geq n\}$, for any $n>T\Vert g\Vert$, we also get 
\[
\sup_{\substack{\tau \leq T \\ \tau\in \mathcal{T}_x}} \mathbb{P}_x[A_n^\tau ]=0.
\]
Consequently, recalling that $\varepsilon>0$ was arbitrary, we obtain
\[\lim_{n\to\infty} \sup_{\tau\in \mathcal{T}_x} \mathbb{E}_x[A_n^\tau e^{\int_0^\tau g(X_s)ds}]=0,
\]
which concludes the proof.
\end{proof}

Next, we consider an optimal stopping problem of the form
\begin{align}
u(x)&:=\sup_{\tau\in \mathcal{T}_{x,b}}\ln \mathbb{E}_x\left[\exp\left(\int_0^\tau g(X_s)ds +G(X_{\tau})\right)\right], \quad x\in E; \label{eq:v_stop}
\end{align}
note that here the non-positivity assumption for $G$ is only a normalisation as for a generic $\Tilde{G}$ we may set $G(\cdot)=\Tilde{G}(\cdot)-\Vert \Tilde{G}\Vert$ to get $G(\cdot)\leq 0$.

The properties of the map~\eqref{eq:v_stop} are summarised in the following proposition.

\begin{proposition}\label{pr:v_stop}
Let the map $u$ be given by~\eqref{eq:v_stop}. Then, $x\mapsto u(x)$ is continuous and bounded. Also, we get
\begin{equation}\label{eq:pr:v_stop:id}
u(x)=\sup_{\tau\in \mathcal{T}_x}\ln \mathbb{E}_x\left[\exp\left(\int_0^\tau g(X_s)ds +G(X_{\tau})\right)\right], \quad x\in E.
\end{equation}
Moreover, the process
\[
z(t):=e^{\int_0^t g(X_s)+u(X_t)}, \quad t\geq 0,
\]
is a supermartingale and the process $z(t\wedge \hat\tau)$, $t\geq 0$, is a martingale, where
\begin{equation}\label{eq:pr:v_stop:optstop}
\hat\tau:=\inf\{t\geq 0: u(X_t)\leq G(X_t)\}.
\end{equation}
\end{proposition}
\begin{proof}
For transparency, we split the proof into two steps: (1) proof of the continuity of $x\mapsto u(x)$ and identity~\eqref{eq:pr:v_stop:id}; (2) proof of the martingale properties of the process $z$.

\medskip 

\textit{Step 1.} We show that the map $x\mapsto u(x)$ is continuous and the identity~\eqref{eq:pr:v_stop:id} holds. For any $T\geq 0$ and $x\in E$, let us define
\begin{align}
\hat{u}(x)&:=\sup_{\tau\in \mathcal{T}_x}\ln \mathbb{E}_x\left[\exp\left(\int_0^\tau g(X_s)ds +G(X_{\tau})\right)\right];\label{eq:eq:v_stop_unb}\\
u_T(x)&:=\sup_{\tau\leq T}\ln \mathbb{E}_x\left[\exp\left(\int_0^\tau g(X_s)ds +G(X_{\tau})\right)\right].\label{eq:v_stop_T}
\end{align}
Using Assumption~\eqref{A3} and following the proof of Proposition 10 and Proposition 11 in~\cite{JelPitSte2019a}, we get that the map $(T,x)\mapsto u_T(x)$ is jointly continuous and bounded; see also Remark 12 therein. We show that $u_T(x)\to \hat{u}(x)$ as $T\to\infty$ uniformly in $x\in E$. 
Noting that 
\[
-\Vert G\Vert\leq u_T(x)\leq u(x), \quad T\geq 0, \, x\in E,
\]
and using the inequality $|\ln y-\ln z|\leq \frac{1}{\min(y,z)}|y-z|$, $y,z>0$, to show $u_T(x)\to \hat{u}(x)$ as $T\to\infty$ uniformly in $x\in E$ it is enough to show $e^{u_T(x)}\to e^{\hat{u}(x)}$ as $T\to\infty$ uniformly in $x\in E$. Then, using Lemma~\ref{lm:bound_r(f)}, for any $\varepsilon>0$, we may find $T\geq 0$ such that for any $x\in E$, we obtain
\begin{align*}
0\leq e^{\hat{u}(x)}-e^{u_T(x)}& \leq \sup_{\tau \in \mathcal{T}_x} \mathbb{E}_x\left[e^{\int_0^\tau g(X_s)ds +G(X_{\tau})}-e^{\int_0^{\tau\wedge T} g(X_s)ds +G(X_{\tau\wedge T})}\right]\nonumber\\
& \leq \sup_{\tau \in \mathcal{T}_x} \mathbb{E}_x\left[1_{\{\tau \geq T \}}\left( e^{\int_0^\tau g(X_s)ds +G(X_{\tau})}-e^{\int_0^{T} g(X_s)ds +G(X_{T})}\right)\right]\nonumber \\
& \leq \sup_{\tau \in \mathcal{T}_x} \mathbb{E}_x\left[1_{\{\tau \geq T \}} e^{\int_0^\tau g(X_s)ds +G(X_{\tau})}\right] \nonumber \\
& \leq \sup_{\substack{\tau\geq T\\ \tau\in \mathcal{T}_{x}}} \mathbb{E}_x\left[e^{\int_0^\tau g(X_s)ds}\right]\leq \varepsilon.
\end{align*}
Thus, letting $\varepsilon\to 0$, we get $e^{u_T(x)}\to e^{\hat{u}(x)}$ as $T\to\infty$ uniformly in $x\in E$ and consequently $u_T(x)\to \hat{u}(x)$ as $T\to\infty$ uniformly in $x\in E$. Thus, from the continuity of $x\mapsto u_T(x)$, $T\geq 0$, we get that the map $x\mapsto \hat{u}(x)$ is continuous. 

Now, we show that $u\equiv \hat{u}$. First, we show that $\lim_{T\to\infty} u_T(x) = \tilde{u}(x)$, where $ \tilde{u}(x):=\sup_{\tau\in \mathcal{T}_x} \liminf_{T\to\infty} \ln \mathbb{E}_x\left[e^{\int_0^{\tau\wedge T} g(X_s)ds +G(X_{\tau\wedge T})}\right]$, $x\in E$. For any $T\geq 0$ and $x\in E$, we get
\begin{equation*}
u_T(x) = \sup_{\tau\leq T} \liminf_{S\to\infty}\ln \mathbb{E}_x\left[e^{\int_0^{\tau\wedge S} g(X_s)ds +G(X_{\tau\wedge S})}\right] \leq  \tilde{u}(x),
\end{equation*}
thus we get $\lim_{T\to\infty} u_T(x) \leq  \tilde{u}(x)$. Also, for any $x\in E$, $\tilde{\tau}\in \mathcal{T}_x$, and $T\geq 0$, we get
\[
\ln \mathbb{E}_x\left[e^{\int_0^{\tilde\tau\wedge T} g(X_s)ds +G(X_{\tilde\tau\wedge T})}\right]\leq u_T(x).
\]
Thus, letting $T\to\infty$ and taking supremum over $\tilde{\tau}\in \mathcal{T}_x$ we get $\lim_{T\to\infty} u_T(x) =  \tilde{u}(x)$, $x\in E$. Also, using the argument from Lemma 2.2 from~\cite{JelSte2020} we get $ \tilde{u} \equiv u$. Thus, we get $u(x) = \lim_{T\to\infty} u_T(x) = \hat{u}(x)$, $x\in E$, hence the map $x\mapsto u(x)$ is continuous. Also, we get~\eqref{eq:pr:v_stop:id}.

\medskip 

\textit{Step 2.} We show the martingale properties of $z$. First, we focus on the stopping time $\hat\tau$. Let us define
\[
\tau_T:=\inf\{t\geq 0: u_{T-t}(X_t)\leq G(X_t)\}.
\]
Using the argument from Proposition 11 in~\cite{JelPitSte2019a} we get that $\tau_T$ is an optimal stopping time for $u_{T}$. Also, noting that the map $T\mapsto u_T(x)$, $x\in E$, is increasing, we get that $T\mapsto \tau_T $ is also increasing, thus we may define $\tilde{\tau}:=\lim_{T\to\infty} \tau_T$. We show that $\tilde\tau\equiv \hat\tau$. 

Let $A:=\{\tilde\tau<\infty\}$. First, we show that $\tilde\tau\equiv \hat\tau$ on $A$.  On the event $A$, we get $u_{T-\tau_T}(X_{\tau_T})=G(X_{\tau_T})$. Thus, letting $T\to\infty$, we get $u(X_{\tilde\tau})=G(X_{\tilde\tau})$, hence we get $\hat\tau\leq \tilde\tau$. Also, noting that $u_S(x)\leq u(x)$, $x\in E$, $S\geq 0$, on the set $\{\hat\tau\leq T\}$ we get $u_{T-\hat\tau}(X_{\hat\tau})\leq u(X_{\hat\tau})\leq G(X_{\hat\tau})$, hence 
\begin{equation}\label{eq:pr:v_stop:1}
\tau_T\leq \hat\tau.
\end{equation}
Thus, recalling that $\hat\tau\leq \tilde\tau < \infty$ and letting $T\to\infty$ in~\eqref{eq:pr:v_stop:1}, we get $\tilde\tau \leq  \hat\tau$, which shows $\tilde\tau \equiv \hat\tau$ on $A$.

Now, we show that $\tilde\tau \equiv \hat\tau$ on $A^c$. Let $\omega\in A^c$ and suppose that $\hat\tau(\omega)<\infty$. Then, we may find $T\geq 0$ such that $\hat\tau(\omega)<T$. Also, for any $S\geq T$ we get
\[
u_{S-\hat\tau(\omega)}(X_{\hat{\tau}(\omega)}(\omega))\leq u(X_{\hat{\tau}(\omega)}(\omega))\leq G(X_{\hat{\tau}(\omega)}(\omega)).
\]
Thus, we get $\tau_S(\omega)\leq\hat\tau(\omega)$ for any $S\geq T$. Consequently, letting $S\to\infty$ we get $ \tilde\tau(\omega)<\infty$, which contradicts the choice of $\omega\in A^c$. Consequently, on $A^c$ we have $\tilde\tau=\infty= \hat\tau$.

Finally, we show the martingale properties. Let us define the processes
\begin{align*}
z_T(t)& :=e^{\int_0^{t\wedge T} g(X_s)ds+ u_{T-t\wedge T}(X_{t\wedge T})}, \quad T,t\geq 0,\\
z(t)&:=e^{\int_0^{t} g(X_s)ds+ u(X_t)}, \quad t\geq 0.
\end{align*}
Using standard argument we get that for any $T\geq 0$, the process $z_T(t)$, $t\geq 0$, is a supermartingale and $z_T(t\wedge \tau_T)$, $t\geq 0$, is a martingale; see e.g.~\cite{Fak1970,Fak1971} for details. Also, recalling that from the first step we get $u_T(x)\to u(x)$ as $T\to \infty$ uniformly in $x\in E$, for any $t\geq 0$, we get that $z_T(t)\to z(t)$ and $z_T(t\wedge \tau_T)\to z(t\wedge\hat\tau)$ as $T\to\infty$. Consequently, using Lebesgue’s dominated convergence theorem, we get that the process $z(t)$ is a supermartingale and $z(t\wedge \hat\tau)$, $t\geq 0$, is a martingale, which concludes the proof.
\end{proof}

Next, we consider an optimal stopping problem in a compact set and dyadic time-grid. More specifically, let $\delta>0$, let $B\subset E$ be compact and assume that $\mathbb{P}_x[\tau_B<\infty]=1$, $x\in B$, where $\tau_B:=\delta\inf\{n \in \mathbb{N}\colon X_{n\delta}\notin B\}$.  Within this framework, we consider an optimal stopping problem of the form
\begin{align}\label{eq:v_stop_B}
u_B(x)&:=\sup_{\tau\in \mathcal{T}^\delta}\ln \mathbb{E}_x\left[\exp\left(\int_0^{\tau\wedge \tau_B} g(X_s)ds +G(X_{\tau\wedge \tau_B})\right)\right], \quad x\in E. 
\end{align}
The properties of~\eqref{eq:v_stop_B} are summarised in the following proposition.
\begin{proposition}\label{pr:stop_families}
Let $u_B$ be given by~\eqref{eq:v_stop_B}. Then, we get
\begin{equation}\label{eq:pr:stop_families:1}
 u_B(x)=\sup_{\tau\in \mathcal{T}_{x,b}^\delta}\ln \mathbb{E}_x\left[\exp\left(\int_0^{\tau\wedge \tau_B} g(X_s)ds +G(X_{\tau\wedge \tau_B})\right)\right], \quad x\in E.   
\end{equation}
Also, the map $x\mapsto u_B(x)$ is continuous and bounded. Moreover, the process
\[
z_\delta(n):=e^{\int_0^{n\delta} g(X_s)+u(X_{n\delta})}, \quad n\in \mathbb{N},
\]
is a supermartingale and the process $z(n\wedge \hat\tau/\delta)$, $n\in \mathbb{N}$, is a martingale, where
\begin{equation}
\hat\tau:=\delta\inf\{n\in \mathbb{N}\colon  u_B(X_{n\delta})\leq G(X_{n\delta})\}.
\end{equation}
\end{proposition}
\begin{proof}
To ease the notation, let us define
\begin{align*}
    \hat{u}_B(x)&:=\sup_{\tau\in \mathcal{T}_{x,b}^\delta}\ln \mathbb{E}_x\left[\exp\left(\int_0^{\tau\wedge \tau_B} g(X_s)ds +G(X_{\tau\wedge \tau_B})\right)\right], \quad x\in E,\\
    u_B^n(x)&:=\sup_{\substack{\tau\in \mathcal{T}^\delta\\ \tau\leq n\delta}}\ln \mathbb{E}_x\left[\exp\left(\int_0^{\tau\wedge \tau_B} g(X_s)ds +G(X_{\tau\wedge\tau_B})\right)\right], \quad n \in \mathbb{N}, \,x\in E,
\end{align*}
    and note that we get $u_B^n(x)\leq \hat{u}_B(x)\leq u_B(x)$, \, $x\in E$. Next, note that using  the boundedness of $G$  and Proposition~\ref{pr:UI}, by Lebesgue’s dominated convergence theorem, we obtain
\[
u_B(x)=\sup_{\tau\in \mathcal{T}}\lim_{n\to\infty}\ln \mathbb{E}_x\left[\exp\left(\int_0^{\tau\wedge (n\delta)\wedge \tau_B} g(X_s)ds +G(X_{\tau\wedge (n\delta)\wedge\tau_B})\right)\right], \quad x\in E.
\]
Also, for any $n\in \mathbb{N}$, $x\in E$, and $\tau\in \mathcal{T}^\delta$, we get
\[
u_B^n(x)\geq \ln \mathbb{E}_x\left[\exp\left(\int_0^{\tau\wedge (n\delta)\wedge \tau_B} g(X_s)ds +G(X_{\tau\wedge (n\delta)\wedge\tau_B})\right)\right], \quad x\in E.
\]
Thus, letting $n\to\infty$ and taking the supremum with respect to $\tau\in \mathcal{T}^\delta$, we get $\lim_{n\to\infty}u_B^n(x)\geq u_B(x)$, $x\in E$. Consequently, we have
\[
\lim_{n\to\infty}u_B^n(x)=\hat{u}_B(x)= u_B(x), \quad x\in E,
\]
which concludes the proof of~\eqref{eq:pr:stop_families:1}.

Let us now show the continuity of the map $ x\mapsto u_B(x)$ and the martingale characterisation. To see this, note that using a standard argument one may show that, for any $n\in \mathbb{N}$ and $ x\in B$, we get
\begin{align*}
    u^0_B(x)&=G(x), x\in B,\\
    e^{u^{n+1}_B(x)}&=\max(e^{G(x)},\mathbb{E}_x\left[1_{\{X_{\delta}\in B\}}e^{\int_0^\delta g(X_s)ds+u^n_B(X_\delta)}+1_{\{X_{\delta}\notin B\}}e^{\int_0^\delta g(X_s)ds+G(X_\delta)}\right],
\end{align*}
and, for any $n\in \mathbb{N}$ and $ x\notin B$, we get $u^n_B(x)=G(x)$; see e.g. Section 2.2 in~\cite{Shi1978} for details. Thus, letting $n\to\infty$, for $x\in B$, we have
\[
 e^{u_B(x)}=\max(e^{G(x)},\mathbb{E}_x\left[1_{\{X_{\delta}\in B\}}e^{\int_0^\delta g(X_s)ds+u_B(X_\delta)}+1_{\{X_{\delta}\notin B\}}e^{\int_0^\delta g(X_s)ds+G(X_\delta)}\right],
\]
while for $x\notin B$, we get $u_B(x)=G(x)$. Also, using Assumption~\eqref{A2}, we get that the process X is strong Feller. Thus, repeating the argument used in Lemma 4 from Chapter II.5 in~\cite{GikSko2004}, we get that, for any bounded and measurable function $h\colon E\mapsto \mathbb{R}$, the map
\[
E\ni x\mapsto\mathbb{E}_x\left[e^{\int_0^\delta g(X_s)ds}h(X_t)\right]
\]
is continuous and bounded. Applying this observation to $h(x):=1_{\{x\in B\}}e^{u_B(x)}$ and $h(x):=1_{\{x\notin B\}}e^{G(x)}$, $x\in E$, we get the continuity of $x\mapsto u_B(x)$. Also, using the argument from Proposition~\ref{pr:_delta_m:martingale} we get that $z_\delta(n)$, $n\in \mathbb{N}$ is a supermartingale and $z(n\wedge \hat\tau/\delta)$, $n\in \mathbb{N}$, is a martingale, which concludes the proof.
\end{proof}

\bibliographystyle{plain}
\bibliography{RSC_bibliografia}

\section*{Statements and Declarations}
Damian Jelito and {\L}ukasz Stettner acknowledge research support by Polish National Science Centre grant no. 2020/37/B/ST1/00463.

The authors have no relevant financial or non-financial interests to disclose.

The authors contributed equally to this work.

\end{document}